\documentclass[10pt]{article}
\usepackage[a4paper]{geometry}
\usepackage{amsmath, amssymb, amsthm}

\usepackage{appendix}
\usepackage{multicol} 
\usepackage{makeidx}
\usepackage{cite}
\usepackage{captdef}
\usepackage{paralist}

\usepackage[colorlinks=true]{hyperref}
\hypersetup{urlcolor=blue, citecolor=red}

\usepackage{psfrag}
\usepackage{graphicx}
\usepackage{graphics}
\usepackage[dvipsnames]{xcolor}

\usepackage[showonlyrefs]{mathtools}
\mathtoolsset{showonlyrefs=true}

\usepackage[in]{fullpage}

\usepackage{subcaption}

    \newtheorem{theorem}{Theorem}[section]
    \newtheorem{coro}[theorem]{Corollary}
    \newtheorem{lemma}[theorem]{Lemma}
    \newtheorem{prop}[theorem]{Proposition}
    \theoremstyle{definition}
    \newtheorem{definition}[theorem]{Definition}
    \newtheorem{remark}[theorem]{Remark}
    
\DeclareMathOperator{\dist}{dist}

\newcommand{\weakto}{\rightharpoonup}

\newcommand{\eps}{\varepsilon} 
\renewcommand{\epsilon}{\varepsilon}
\renewcommand{\a}{\alpha}
\renewcommand{\b}{\beta}
 
\newcommand{\N}{{\mathbb{N}}}

\newcommand{\R}{{\mathbb{R}}}

\numberwithin{equation}{section}

\title{An asymptotic relationship between Lane-Emden systems and the 1-bilaplacian equation}
\author{Nicola Abatangelo\thanks{N. Abatangelo is partially supported by the GNAMPA-INdAM project ``Equazioni nonlocali di tipo misto e geometrico'' (Italy) and the PRIN project 2022R537CS ``$NO^3$ - Nodal Optimization, NOnlinear elliptic equations, NOnlocal geometric problems, with a focus on regularity'' (Italy).}, Alberto Saldaña\thanks{A. Saldaña is supported by CONACYT grant A1-S-10457 (Mexico) and by UNAM-DGAPA-PAPIIT grant IA100923 (Mexico).}, and Hugo Tavares \thanks{Hugo Tavares is partially supported by the Portuguese government through FCT - Funda\c c\~ao para a Ci\^encia e a Tecnologia, I.P., through CAMGSD, IST-ID (grant UID/MAT/04459/2020) and through the project NoDES (PTDC/MAT-PUR/1788/2020).}}

\date{\today}

\begin{document}
\maketitle

\begin{abstract}
Consider the following Lane-Emden system with Dirichlet boundary conditions: 
\[
-\Delta U = |V|^{\b-1}V,\ 
-\Delta V = |U|^{\a-1}U  \text{ in }\Omega,\qquad 
U=V= 0  \text{ on }\partial \Omega,
\]
in a bounded domain $\Omega$, for $(\alpha,\beta)$ subcritical. We study the asymptotic behavior of least-energy solutions when $\beta\to \infty$, for any fixed $\alpha$ which, in the case $N\geq 3$, is smaller than $2/(N-2)$. We show that these solutions converge to least-energy solutions of a semilinear equation involving the 1-bilaplacian operator, establishing a new relationship between these objects. As a corollary, we deduce the asymptotic behavior of solutions to $p$-bilaplacian Lane-Emden equations as the power in the nonlinearity goes to infinity.

For the proof, we rely on the reduction by inversion method and on tools from nonsmooth analysis,  considering an auxiliary nonlinear eigenvalue problem. We characterize its value in terms of the Green function, and prove a Faber-Krahn type result. In the case of a ball, we can characterize explicitly the eigenvalue, as well as the limit profile of least-energy solutions to the system as $\b\to\infty$.
\end{abstract}

\noindent \\
\begin{MScodes} 
 	35B40, 35G15, 35J30, 35J47, 35P30, 49J52, 49Q10
\end{MScodes}

\section{Introduction}

In this paper we study the asymptotic behavior of least-energy solutions of the following Lane-Emden system 
\begin{align}\label{hsystem:intro}
-\Delta U_\b = |V_\b|^{\b-1}V_\b,\ 
-\Delta V_\b = |U_\b|^{\a-1}U_\b  \text{ in }\Omega,\qquad 
U_\b=V_\b= 0 & \text{ on }\partial \Omega,
\end{align}
where $\Omega$ is a bounded domain in $\R^N$ with $N\geq 1$ and $\a>0,\b>0$ satisfy 
\begin{equation}\label{subcritical:intro}
\alpha\cdot \beta\neq 1\qquad  \text{ and } \qquad \frac{1}{\a+1}+\frac{1}{\b+1} > \frac{N-2}{N},
\end{equation}
that is, the pair $(\alpha,\beta)$ is below the critical hyperbola. The case $\alpha\cdot \beta=1$ is different in nature, being an eigenvalue problem (see \cite{Montenegro}), which we do not consider here. We are mainly interested in characterizing the limit of $(U_\b,V_\b)$ as the power $\b$ tends to infinity, for any fixed $\alpha>0$ if $N=1,2$, or $\alpha\in (0,\frac{2}{N-2})$ if $N\geq 3$.

\begin{figure}[ht]
   \centering\includegraphics[height=6cm]{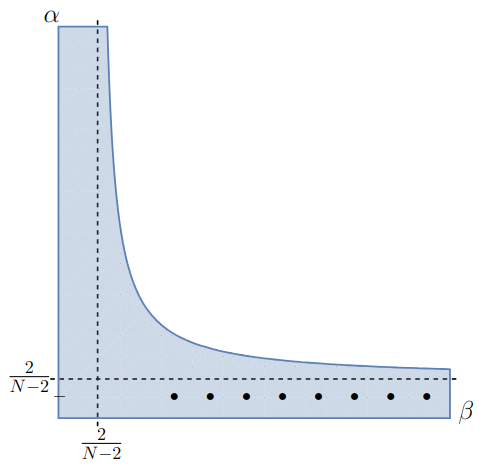}
  \caption{$(\a,\b)$ subcritical with $\a\in(0,\frac{2}{N-2})$ fixed and $\b\to\infty$.}
  \label{chfig}
    \end{figure}

The existence of least-energy solutions of \eqref{hsystem:intro} is well known and several methods for this (and to study qualitative properties) have been developed, see the survey \cite{BMT14}.  In particular, the reduction-by-inversion method establishes a bijection between classical solutions of \eqref{hsystem:intro} and critical points of the energy functional $J_{\a,\b}:W^{2,\frac{\beta+1}{\beta}}(\Omega)\cap W_0^{1,\frac{\beta+1}{\beta}}(\Omega) \to \R$ given by 
\begin{align}\label{J:intro}
J_{\a,\b}(U)
:=\frac{\b}{\b+1}\int_\Omega |\Delta U|^\frac{\b+1}{\b}\, dx - \frac{1}{\a+1}\int_\Omega |U|^{\a+1}\, dx.
\end{align}
More concretely, given $U\in W^{2,\frac{\beta+1}{\beta}}(\Omega)\cap W^{1,\frac{\beta+1}{\beta}}_0(\Omega)$ and $V:=-|\Delta U|^{\frac{1}{\beta}-1}\Delta U$, then $(U,V)$ is a classical solution of \eqref{hsystem:intro} if, and only, if, $J_{\alpha,\beta}'(U)=0$ (see, for instance, \cite[Lemma 4.8]{BMT14}, whose proof can be found in \cite[Theorem 1.1]{S08} and \cite[Appendix A]{BSR12}). Observe that a critical point of $J_{\a,\b}$ is a solution of the Navier $\frac{\b+1}{\b}$-bilaplacian equation
\begin{align}\label{bplap}
\Delta(|\Delta u|^{\frac{1}{\b}-1}\Delta u)=|u|^{\a-1}u\quad \text{ in }\Omega,\qquad u=\Delta u=0\quad \text{ on }\partial \Omega.
\end{align}
Within this framework, one says that $(U_\b,V_\b)\neq (0,0)$ is a least-energy solution of \eqref{hsystem:intro} if 
\begin{align*}
J_{\a,\b}(U_\b)=c_{\alpha,\beta}&:=\inf\{J_{\a,\b}(U):\ U \text{ is a nontrivial critical point of } J_{\alpha,\beta}\}\\
        &=\min\{J_{\a,\b}(U):(U,V)\text{ is a nontrivial solution of \eqref{hsystem:intro}}\}.
\end{align*}
We point out that \eqref{subcritical:intro} implies that the embedding $W^{2,\frac{\beta+1}{\beta}}(\Omega)\hookrightarrow L^{\a+1}(\Omega)$ is compact, and that the homogeneity on the left- and right-hand sides of \eqref{bplap} do not coincide. For other (equivalent) characterizations of the least energy level, see, for instance, \cite[Corollary 2.5, Theorem 3.8 and Propositions 5.4 and 5.12]{BMT14}.

Our asymptotic analysis establishes a new link between the limit profiles (as $\beta\to\infty$) of least-energy solutions of the Hamiltonian system \eqref{hsystem:intro} and the (Navier) 1-bilaplacian equation
\begin{align}\label{1bi:intro}
   \Delta\left(\frac{\Delta u}{|\Delta u|}\right) = |u|^{\a-1}u\quad \text{in }\Omega,\qquad u=\frac{\Delta u}{|\Delta u|}=0 \text{ on } \partial \Omega,
\end{align}
where for any fixed $\alpha>0$ if $N=1,2$, or $\alpha\in (0,\frac{2}{N-2})$ if $N\geq 3$. Equation \eqref{1bi:intro} is a formalism to represent a critical point of a nonsmooth variational problem.  Indeed, note that \eqref{1bi:intro} cannot be understood in a pointwise sense: $\Delta u/|\Delta u|$ is constant (and, at least formally, harmonic) on $\{\Delta u\neq 0\}$, while it is not well defined on $\{\Delta u=0\}$. This suggests that $\Delta u$ should not be thought of as a function, but rather as a \emph{distribution}, and this allows to give a consistent notion for a solution of \eqref{1bi:intro}. To be more precise, following \cite{PRT12}, let
\begin{align}\label{eq:BL0}
BL_{0}(\Omega):=\big\{u \in W_{0}^{1,1}(\Omega): |\Delta u|_T <\infty\big\},
\end{align}
where  
\begin{align*}
|\Delta u|_T=\sup \bigg\{\int_{\Omega} u \Delta \varphi : \varphi \in C^\infty_{c}(\Omega),\ |\varphi|_{\infty} \leq 1\bigg\}
\end{align*}
is the total variation of the \emph{Radon measure} $\Delta u$, see \cite[Section 2]{PRT12} and Lemma~\ref{alt:lem}. We say that $u\in BL_0(\Omega)$ is a solution of \eqref{1bi:intro} if there is $v\in W^{1,1}_0(\Omega)\cap L^\infty(\Omega)$ such that 
\begin{align}\label{v:intro}
-\Delta v = |u|^{q-2}u \quad \text{ in $\Omega$ (in distributional sense)},\quad |v|_\infty\leq 1, \quad \text{ and }\quad \int_\Omega u(-\Delta v) = |\Delta u|_T.
\end{align}
In this sense, the function $v$ gives a consistent meaning to the quotient $-\Delta u/|\Delta u|$. 

  The \emph{energy} of a solution $u\in BL_0(\Omega)$ of 
\eqref{1bi:intro} is given by 
\begin{align*}
\Phi(u)=|\Delta u|_T-\frac{1}{\a+1}|u|_{\a+1}^{\a+1},\qquad |u|_{\a+1}:=\left(\int_\Omega |u|^{\a+1}\right)^\frac{1}{\a+1},
\end{align*}
and we call $u$ a \emph{least-energy} solution of \eqref{1bi:intro} if $\Phi(u)=\min\{\Phi(w): w \text{ is a nontrivial solution of } \eqref{1bi:intro}\}.$  The functional $\Phi$ can be treated variationally using techniques from nonsmooth analysis, see Section~\ref{sec:1-bilaplacian}.

The following result characterizes the limit profiles and establishes a bridge between \eqref{hsystem:intro} and \eqref{1bi:intro}.  Let $G_{\Omega}$ denote the Dirichlet Green function for the Laplacian  in $\Omega$. Denote by
\begin{equation}\label{eq:1*1**}
1^*=\begin{cases}
    \frac{N}{N-1} & \text{ if } N\geq 2\\
    \infty & \text{ if } N=1
\end{cases},\qquad 1^{**}=\begin{cases}
    \frac{N}{N-2} & \text{ if } N\geq 3\\
    \infty & \text{ if } N=1,2
\end{cases}.
\end{equation}

\begin{theorem}\label{Uq:conv:thm:intro}
Let $\Omega\subset \R^N$ be a bounded domain of class $C^{2,\gamma}$ for some $\gamma\in(0,1]$. Let $\a\in(0,1^{**}-1)$, $\b>0$, and let $(U_\b,V_\b)$ be a least-energy solution of \eqref{hsystem:intro}. There is a least-energy solution $U_\infty\in BL_0(\Omega)$ of \eqref{1bi:intro} such that, up to a subsequence, as $\b\to\infty$,
\begin{align}
U_\b\to U_\infty\quad \text{in $W^{1,r}_0(\Omega)$} \text{ for every } r\in[1,1^*),\qquad \text{ and } \qquad |\Delta U_\b|_1\to |\Delta U_\infty|_T.
\end{align}
Furthermore, let $V_\infty\in W^{2,1+\frac1\a}(\Omega)\cap W_0^{1,1+\frac1\a}(\Omega)$ be the unique solution of
\[
-\Delta V_\infty= |U_\infty|^{\alpha-1}U_\infty \text{ in } \Omega,\qquad V_\infty=0 \text{ on } \partial \Omega.
\]
Then, up to a subsequence, as $\b\to\infty$
\begin{align*}
    V_\b\to V_\infty\qquad \text{weakly in }W^{2,1+\frac{1}{\alpha}}(\Omega),\text{ strongly in } W^{1,\eta}_0(\Omega) \text{ and } C^{m,\sigma}(\overline \Omega),
\end{align*}
for some $\eta>1^*$ and $\sigma\in (0,1)$ depending on $\alpha$, with $m=1$ if $\alpha<1^*-1$, $m=0$ if $\alpha\geq 1^*-1$.
\end{theorem}

If $\Omega$ is the unit ball $B_1$, we can obtain more explicit formulas. We present them for $N\geq 3$ only, but it is possible to do the same for $N=1,2$.

\begin{theorem}\label{ball:thm}
Consider $N\geq 3$. Under the assumptions and notations of Theorem~\ref{Uq:conv:thm:intro}, it holds
\begin{align}\label{k:def}
U_\infty(x) = A_{N,\a}(|x|^{2-N}-1) \quad \text{for }x\in B_1, 
\qquad\text{where}\quad
A_{N,\a} := \left(\frac{2N\, \Gamma \left(\frac{2}{N-2}\right)}
{\Gamma (\alpha +2) \Gamma \left(\frac{2}{N-2}-\alpha \right)}
\right)^{\frac1\a}. 
\end{align}
Moreover, the limit $(U_\infty,V_\infty)$ is unique, and the convergence $(U_\b,V_\b)\to (U_\infty,V_\infty)$ holds without necessarily considering subsequences.
\end{theorem}

\begin{remark}
    For the exact statement regarding the convergence of $V_\beta$ to $V_\infty$ in Theorems~\ref{Uq:conv:thm:intro} and~\ref{ball:thm}, see Corollary~\ref{Vq:conv:coro} below. 
\end{remark}

We remark that Theorems~\ref{Uq:conv:thm:intro} and~\ref{ball:thm} can also be interpreted in the setting of the Navier $\frac{\a+1}{\a}$-bilaplacian Lane-Emden equations when the power of the nonlinearity goes to infinity.
\begin{coro}\label{coro:intro}
Let $\Omega\subset \R^N$ be a bounded domain of class $C^{2,\gamma}$ for some $\gamma\in(0,1)$. Let $\a\in(0,1^{**}-1)$, $\b>0$, and let $v_\b$ be a least-energy solution of
\begin{align}\label{pbilap}
\Delta(|\Delta v_\b|^{\frac{1}{\a}-1}\Delta v_\b)=|v_\b|^{\b-1}v_\b\quad \text{ in }\Omega,\qquad v_\b=\Delta v_\b=0\quad \text{ on }\partial \Omega.
\end{align}
Then $v_\b\to V_\infty$ as $\beta\to \infty$, where the convergence and $V_\infty$ are as in Theorem~\ref{Uq:conv:thm:intro}.  
\end{coro}
This result is new even for $N \leq 3$ and $\alpha=1$, where \eqref{pbilap} reduces to the Navier bilaplacian equation 
\begin{align}\label{bilap:int}
\Delta^2 v_\b=|v_\b|^{\b-1}v_\b\quad \text{ in }B_1,\qquad v_\b=\Delta v_\b=0\quad \text{ on }\partial B_1.
\end{align}
It is interesting to note that the behavior described in Corollary~\ref{coro:intro} for \eqref{bilap:int} is different from the results obtained in \cite{AG06}, where the asymptotic profile of the solution $v_\beta$ of \eqref{bilap:int} is studied as $\b\to\infty$ in dimension $N=4.$  We shall return to this issue later in this introduction  when we discuss the previous literature. 

Next, we state Theorem~\ref{Uq:conv:thm:intro} in the setting of the Navier $\frac{\b+1}{\b}$-bilaplacian equation \eqref{bplap}.
\begin{coro}\label{coro:intro:2}
Let $\Omega\subset \R^N$ be a bounded domain of class $C^{2,\gamma}$ for some $\gamma\in(0,1)$. Let $\a\in[0,1^{**}-1)$,  and let $u_{\beta}$ be a least-energy solution of
\begin{align*}
\Delta(|\Delta u_{\beta}|^{\frac{1}{\b}-1}\Delta u_{\beta})=|u_{\beta}|^{\a-1}u_\beta\quad \text{ in }\Omega,\qquad u_\beta=\Delta u_\beta=0\quad \text{ on }\partial \Omega.
\end{align*}
Then, up to a subsequence and with $U_\infty$ as in Theorem~\ref{Uq:conv:thm:intro}, $u_\beta\to U_\infty$ in $W^{1,r}(\Omega)$ as $n\to\infty$ for $r\in[1,1^*)$ and $|\Delta u_\beta|_1\to |\Delta U_\infty|_T$. Moreover, if $\Omega$ is the unit ball, then the convergence holds with $U_\infty$ as in Theorem~\ref{ball:thm}.
\end{coro}

To show Theorems~\ref{Uq:conv:thm:intro} and~\ref{ball:thm} we consider an auxiliary nonlinear eigenvalue problem, which is of independent interest.  To be more precise, let $p,q$ be such that
\begin{equation}\label{eq:lowerboundbeta:intro}
q\geq  1,\qquad p\geq 1, \qquad (N-2p)q<Np,
\end{equation}
and let
\begin{align}\label{alphaq:intro}
    \Lambda_{p,q}:=\inf_{u\in W^{2,p}_\Delta(\Omega)\setminus\{0\}}\frac{|\Delta u|_p}{|u|_q}
    =\inf\left\{|\Delta u|_p:u\in W^{2,p}_\Delta(\Omega),\ |u|_q=1\right\},
\end{align}
where 
\begin{align}\label{def:Xp:intro}
W^{2,p}_\Delta(\Omega):=\big\{u\in W^{1,p}_0(\Omega): \Delta u\in L^p(\Omega)\big\}.
\end{align}
 We summarize in the next theorem our main results regarding this auxiliary problem, since they are of independent interest.
\begin{theorem}\label{thm:nonlineareigenvalue_pq}
Let $\Omega$ be convex or of class $C^{1,\gamma}$ for some $\gamma\in (0,1]$. Then:
\begin{enumerate}
\item[(a)] if $p>1$ and $(N-2p)q<Np$, then $\Lambda_{p,q}$ is achieved by a function $u_{p,q}\in W^{2,p}_\Delta (\Omega)$ with $|u_{p,q}|_q=1$, and
\begin{align*}
\Delta(|\Delta u_{p,q}|^{p-2}\Delta u_{p,q})=\Lambda_{p,q}|u_{p,q}|^{q-2}u_{p,q}\quad \text{ in }\Omega,\qquad u_{p,q}=\Delta u_{p,q}=0\quad \text{ on }\partial \Omega.
\end{align*}
    \item[(b)]  if $p=1$ and $q\in[1,1^{**})$, then
\begin{equation*}
\Lambda_{1,q}= \inf_{u\in BL_0(\Omega)\backslash\{0\}}\frac{|\Delta u|_T}{|u|_q},
\end{equation*}
and $\Lambda_{1,q}$ is achieved by a function $u_{1,q}\in BL_0(\Omega)$ with $|u|_q=1$. Moreover,
\begin{equation*}
\Delta\left(\frac{\Delta u_{1,q}}{|\Delta u_{1,q}|}\right) = \Lambda_{1,q}|u_{1,q}|^{q-2}u_{1,q}\quad \text{in }\Omega,\qquad u_{1,q}=\frac{\Delta u_{1,q}}{|\Delta u_{1,q}|}=0 \text{ on } \partial \Omega.
\end{equation*}
\end{enumerate}
Moreover, for $(p,q)$ satisfying \eqref{eq:lowerboundbeta}, the map
\[
(p,q)\mapsto \Lambda_{p,q}
\]
is continuous  and, for $q\in [1,1^{**})$ and  $(p_n,q_n)\to (1,q)$, given $u_n$ a  minimizer for $\Lambda_{p_n,q_n}$ with $|u_n|_{q_n}=1$, then there is a minimizer $u_{1,q}$ of $\Lambda_{1,q}$, with $|u_1|_q=1$, such that, up to a subsequence:
\[
  \text{$u_n\to u$ in $W^{1,r}_0(\Omega)$ for every $r\in [1,1^*)$ and $|\Delta u_n|_T\to |\Delta u|_T$.}
\]
\end{theorem}

We make some comment on the proof of this theorem, as well as its connection with Theorem~\ref{Uq:conv:thm:intro}. Regarding items (a) and (b), if $p>1$, then $W^{2,p}_\Delta(\Omega)=W^{2,p}(\Omega)\cap W^{1,p}_0(\Omega)$, but $W^{2,1}(\Omega)\cap W^{1,1}_0(\Omega)$ is strictly contained in $W^{2,1}_\Delta(\Omega)$.  Furthermore, the space $W^{2,1}(\Omega)$ is not reflexive, and in particular one cannot guarantee that a minimizing sequence has a weakly convergent subsequence.  A consequence of these facts is that $\Lambda_{p,q}$ is achieved in $W^{2,p}_\Delta(\Omega)$ for $p>1$, but $\Lambda_{1,q}$ is not achieved in $W^{2,1}_\Delta(\Omega)$; however, we show that, for $q<1^{**}$, $\Lambda_{1,q}$ is achieved in the larger space $BL_0(\Omega)$ (see Lemma~\ref{alpha1:lem} for more details). This is due to the compact embeddings of $BL_0(\Omega)$ into $L^q(\Omega)$, see Proposition~\ref{compact:prop}. We point out that the case $p=1$, $q=1$ has been previously shown in \cite{PRT12}, and we refer to that paper for the exact notion of solution in this situation.

The advantage of considering \eqref{alphaq:intro} is that a minimizer $u_{p,q}$ for $\Lambda_{p,q}$ yields, up to a multiplication by a constant, a least-energy critical point of \eqref{J:intro} with $p=\frac{\beta+1}{\beta}$ and $q=\alpha+1$; in particular, $(U_\b,V_\b)=(u_{p,q},-|\Delta u_{p,q}|^{p-2}\Delta u_{p,q})$ is a solution of \eqref{hsystem:intro}.  Similarly, a minimizer of $\Lambda_{p,q}$ yields, up to a multiplication by a constant, a least energy solution of \eqref{1bi:intro}, see Proposition~\ref{prop:existence_of_les}.  The continuity of the map $(p,q)\mapsto \Lambda_{p,q}$ proved in Theorem~\ref{thm:nonlineareigenvalue_pq}, and the convergence of the associated minimizers, is the main tool to prove the convergence stated in Theorem~\ref{Uq:conv:thm:intro}.  On the other hand, in particular, it also provides a convergence of the nonlinear value problem $\Lambda_{p,q}$ to the linear one $\Lambda_{1,1}$ (see Remark~\ref{rmk:convergencePRT} for more details). We point out that our work also provides a rigorous meaning to the formal limit
\[
\Delta^2_1 u:=\Delta \left(\frac{\Delta u}{|\Delta u|}\right)=\lim_{p\to 1} \Delta (|\Delta u|^{p-2}\Delta u)=:\lim_{p\to 1} \Delta^2_p.
\]
This complements the work \cite{inftybilaplacian}, where the authors study the behavior of $\Lambda_{p,p}$ as $p\to \infty$, in connection with the $\infty$-bilaplacian.
\smallbreak

While minimizers of $\Lambda_{p,q}$ are well characterized for $p>1$ (see for instance \cite[Section 4]{BMT14} and references therin), this is not the case for $p=1$. In this paper we also provide a characterization of minimizers of $\Lambda_{1,q}$. To be more precise, recall that $G_\Omega(\cdot,\cdot)$ denotes the Dirichlet Green function of the Laplacian in $\Omega$ and that $|\cdot|_q$ denotes the $L^q$-norm.

\begin{theorem}\label{char:thm}
Let $\Omega$ be convex or of class $C^{1,\gamma}$ for some $\gamma\in (0,1]$ and $q\in [1,1^{**}-1)$.  Then,
\begin{align*}
\Lambda_{1,q}(\Omega)=\frac{1}{|G_\Omega(\cdot,x_M)|_q},
\end{align*}
where $x_M\in \Omega$ is such that
\begin{align}\label{xMdef}
    |G_\Omega(\cdot,x_M)|_q=\max_{x\in \overline{\Omega}}|G_\Omega(\cdot,x)|_q.
\end{align}
Moreover, the function
\begin{align}\label{exp:for}
u_1(x)=\frac{G_\Omega(x,x_M)}{|G_\Omega(\cdot,x_M)|_q}\qquad \text{ for }x\in \Omega
\end{align}
achieves $\Lambda_{1,q}$, namely, $u_1\in BL_0(\Omega)$, $|u_1|_q=1$, and $|\Delta u_1|_T=\Lambda_{1,q}$.  If the maximum point $x_M$ is unique, then $u_1$ is, up to a multiplicative constant, the only minimizer achieving  $\Lambda_{1,q}$.

In general, any other possible function $u$ achieving $\Lambda_{1,q}$ is, up to sign, positive in $\Omega$ and $\mu:=-\Delta u$ is a positive Radon measure. 
\end{theorem}
This is an extension of \cite[Theorem 1.2]{PRT12}, which considered the case $q=1$; in this case, $|G_\Omega(\cdot,x)|_1=\Psi(x)$, where $\Psi$ is the torsion function of $\Omega$.  Note that $\Lambda_{p,q}(\Omega)$ provides the best constant for the continuous (and compact) embedding $W^{2,p}_\Delta(\Omega)\hookrightarrow L^q(\Omega)$ and also for $BL_0(\Omega)\hookrightarrow L^q(\Omega)$.  The explicit value of $\Lambda_{1,q}(B_1)$ is computed in Theorem~\ref{eigenball:thm} (see also Remark~\ref{L11:rmk} regarding $\Lambda_{1,1}(B_1)$). 

As stated in the previous theorem, when the maximum point $x_M$ (see \eqref{xMdef}) is unique, then $\eqref{exp:for}$ is the unique minimizer of $\Lambda_{1,q}$, up to sign and  normalization in $L^q$ sense (see Proposition~\ref{prop:uniqueness} for the proof of this particular statement).  Using Talenti's comparison principle, we can show that $x_M=0$ is a maximum point in the case of the unit ball $B_1$ centered at the origin (Proposition~\ref{uniquexm}), which yields the following result. 
\begin{theorem}\label{eigenball:thm}
The function $G_{B_1}(x,0)$ is a  minimizer of $\Lambda_{1,q}(B_1)$. In particular, for $N\geq 3$ we have that $|x|^{2-N}-1$ achieves $\Lambda_{1,q}(B_1)$ for any $q\in [1,\frac{N}{N-2})$, and 
\begin{align*}
    \Lambda_{1,q}(B_1)=\frac{4\pi^{N/2}}{\Gamma(\frac{N}2-1)}\bigg(
    \frac{\Gamma(\frac{N}{N-2})\Gamma(\frac{N}2+1)}{\pi^{N/2}\Gamma(\frac{N}{N-2}-q)\,\Gamma(q+1)}\bigg)^{\frac1q}.
\end{align*}
\end{theorem}
It is interesting to note that the minimizer $|x|^{2-N}-1$ of $\Lambda_{1,q}(B_1)$ is independent of $q$.  Here, the symmetries of the ball play an important role.  However, in more general domains, the concentration point $x_M$ might also depend on $q$ and one would obtain different minimizers for every $q\in [1,\frac{N}{N-2})$, see Remark~\ref{rmk:q} in this regard.   We can also show the following Faber-Krahn-type inequality result.
\begin{prop}\label{faber-krahn}
For any $\Omega\subseteq\R^N$ convex or of class $C^{1,\gamma}$ for some $\gamma\in(0,1]$, and such that $|\Omega|=|B_1|$, it holds
\begin{align*}
\Lambda_{1,q}(B_1)\leq\Lambda_{1,q}(\Omega).
\end{align*}
\end{prop}

The results in this paper complement other asymptotic characterizations that have been previously studied in the literature for systems and equations. First, we mention the case of the Lane-Emden equation
\begin{align}\label{le:eq}
    -\Delta u_\b = |u_\b|^{\b-1}u_\b\quad \text{ in }\Omega,\qquad u_\b=0\quad \text{ on }\partial \Omega.
\end{align}
Note that \eqref{subcritical:intro} reduces to \eqref{le:eq} when $\a=\b$ and $U_\b=V_\b>0$. If $N\geq 3$, then $0<\b+1<2^*=\frac{2N}{N-2}$ and therefore one cannot consider the asymptotic behavior of solutions as $\b\to \infty.$ However, if $N=2$, then \eqref{le:eq} is subcritical for all $\b>1$ and the asymptotic profile of the solution $u_\b$ is well understood.  The interest in these characterizations started with the seminal works by Ren and Wei \cite{ren1994two,ren1996single}. It is known that least-energy solutions exhibit a single concentration point phenomenon (solutions go to zero locally uniformly outside the concentration point as $\b\to \infty$). Furthermore, these solutions do \emph{not} blow-up, they remain uniformly bounded and $|u_\b|_{\infty}:=\sup_\Omega|u_\b|\to \sqrt{e}$ as $p\to \infty,$ see \cite{G04}.  If $\Omega$ is a ball, then a sharp asymptotic profile is known, namely
\begin{align*}
    \b u_\b\to 4\sqrt{e}\log\frac{1}{|x|}\qquad \text{ in }C^1(\overline{\Omega}\backslash\{0\}) \text{ as }\beta\to\infty,
\end{align*}
see  \cite{G04} (see also \cite{ianni2021sharp} for a similar result regarding nodal solutions).

If $\a=1$, then \eqref{hsystem:intro} reduces to the subcritical Navier bilaplacian equation \begin{align}\label{bi:eq}
    \Delta^2 v_\b = |v_\b|^{\b-1}v_\b\quad \text{ in }\Omega,\qquad v_\b=\Delta v_\b=0\quad \text{ on }\partial \Omega.
\end{align}
When $\Omega\subset \R^4$ is a bounded and smooth domain, an asymptotic analysis as $\b\to \infty$ is done in \cite{AG06}; in particular, it is shown that, up to a subsequence, there is $x_0\in \Omega$ (which is a critical point of the Robin function) such that $\b v_\b\to C G(\cdot,x_0)$ in $C^4(\Omega \backslash \{x_0\})$, where $C>0$ is an explicit constant and $G$ is the Navier Green function of the bilaplacian in $\Omega.$  Note that this is a very different behavior with respect to the one described in Corollary~\ref{coro:intro} for $N\leq 3$ and $\alpha=1$.  In particular, the limit profile in the case $N=4$ is unbounded in $L^\infty(\Omega)$, whereas the one in $N=3$ is a function in $W^{2,2}(\Omega)\cap W_0^{1,2}(\Omega)\cap C^{0,\sigma}(\overline \Omega)$.   
This difference is mainly due to the fact that the Sobolev critical exponent for the bilaplacian is $2^{**}=\frac{2N}{N-4}$ and $N=4$ is the transition between $2^{**}$ being finite or infinite. 

Regarding the Lane-Emden system \eqref{le:eq}, the following results are known. In \cite{guerra2007asymptotic} the author considers the case $\a=\frac{2}{N-2}$ ($N\geq 3$) and $\beta \to \infty$; namely, $(\a,\b)$ approaches the asymptote of the critical hyperbola (portrayed in Figure~\ref{chfig}).  The main result in \cite{guerra2007asymptotic} shows that, for a smooth convex and bounded domain $\Omega\subset \R^N$ with $N\geq 3$, if there is only one concentration point $x_0\in \Omega$, then (up to a subsequence)
\begin{align*}
\frac{U_\b}{\left(\int_\Omega|U_\b|_{\b}^\b\right)^{{2}{N-2}}}\to W\qquad \text{ in }C^2(\Omega \backslash \{x_0\})\text{ as }\b\to \infty,
\end{align*}
where $W$ is a solution of $-\Delta W = G_\Omega^{\frac{2}{N-2}}(x,x_0)$ in $\Omega$ and $W=0$ on $\partial \Omega$ and $G_\Omega$ is the Dirichlet Green function of Laplacian in $\Omega$.  This is also a different behavior compared to the description in Theorems~\ref{Uq:conv:thm:intro} and~\ref{ball:thm}.

For $N=2$, the behavior of positive solutions of \eqref{hsystem:intro} in smooth bounded domains when both $\a$ and $\b$ tend to infinity has also been studied in some particular cases in \cite{KS23,ZHW22}.  In \cite{ZHW22} it is shown that if $\a=\b+\theta_\b$ with $\theta_\b$ and $\theta_\b\to \theta$ as $\b\to \infty$, then there is a finite set of concentration points $S=\{x_1,\ldots,x_k\}\subset \Omega$ such that, up to a subsequence,
\begin{align*}
\b U_\b,\b V_\b \to 8\pi\sqrt{e}\sum_{i=1}^k G_\Omega(x,x_i)\qquad \text{ in }C^2_{loc}(\overline{\Omega}\backslash S).
\end{align*}
The concentration points $x_i$ are also characterized, see \cite{ZHW22}.  On the other hand, in \cite{KS23}, positive solutions of \eqref{hsystem:intro} in planar bounded $C^2$ domains are shown to be uniformly bounded as $\b\to \infty$ whenever 
\begin{align}\label{K}
K^{-1}\a\leq \b\leq K\a\qquad \text{ for some $K>1$.}
\end{align}
In \cite{KS23} it is also shown that \eqref{K} cannot be removed, since it is shown that, if $\Omega$ is a disc and $\a=1$, then any positive solution $(U_\b,V_\b)$ of \eqref{hsystem:intro} satisfies that $C^{-1}\log(\b)\leq |U_\b|_\infty\leq C\log(\b)$ as $\b\to \infty.$

We also mention that the limit profiles of solutions of  \eqref{hsystem:intro} as $(\alpha,\beta)$ goes to a point  $(\a_0,\b_0)$ on the critical hyperbola (namely, $\frac{1}{\a_0+1}+\frac{1}{\b_0+1} = \frac{N-2}{N}$) has been studied in \cite{CK19,G08}.  In this case, both components exhibit a blow-up (concentration) behavior at critical points of the Robin function. Furthermore, a suitable rescaling of this solution converges to a solution of the Lane-Emden system in $\R^N$.  Although these limit profiles are not explicit, their decay rate at infinity are known with precision can be characterized.

\medskip

The paper is organized as follows. In Section~\ref{eigen:sec} we study the existence of minimizers for the auxiliary nonlinear eigenvalue problem \eqref{alphaq:intro}, its link with solutions of \eqref{1bi:intro} and the proof of Theorems~\ref{thm:nonlineareigenvalue_pq} and~\ref{char:thm}.  The proof of Theorem~\ref{Uq:conv:thm:intro} can be found in Section~\ref{sec:Hamiltonian}, while Section~\ref{ball:sec} is devoted to the proof of Theorems~\ref{ball:thm} and~\ref{eigenball:thm} and of Proposition~\ref{faber-krahn}.  Finally, we include a self-contained appendix with several known useful results regarding the space $BL_0(\Omega)$.

\subsection{Notation}\label{sec:notation}
Let $\Omega\subset\R^N$ be a bounded Lipschitz domain of $\R^N$ ($N\geq 3$)
and let $\a,\b\in\R$ be positive and  subcritical, that is,
\begin{equation}\label{subcritical}
\a,\b>0,\qquad \frac{1}{\a+1}+\frac{1}{\b+1} > \frac{N-2}{N},
\end{equation}
see \cite{ClementdeFigueiredoMitidieri,ClementvanderVorst}. Let 
\begin{align*}
    |u|_{t}=\left(\int_\Omega |u|^{t}\right)^\frac{1}{t} \text{ for $t\geq 1$},\qquad |u|_\infty=\sup_{\Omega} |u|.
\end{align*}
For $p\geq 1$, we denote
\begin{align}\label{def:Xp}
W^{2,p}_\Delta(\Omega):=\big\{u\in W^{1,p}_0(\Omega): \Delta u\in L^p(\Omega)\big\}.
\end{align}
Under additional geometrical or regularity assumptions on $\Omega$ (see Lemma~\ref{equiv:lem} below for the details), this is a Banach space when endowed with the norm $|\Delta \cdot|_p$. Moreover, 
\begin{align*}
W^{2,p}_\Delta(\Omega)=W^{2,p}(\Omega)\cap W^{1,p}_0(\Omega)\qquad\text{for }p>1,
\end{align*}
and the norm $|\Delta \cdot |_p$ is equivalent to the standard $W^{2,p}(\Omega)$ one. However, $W^{2,1}(\Omega)\cap W^{1,1}_0(\Omega) \subsetneq W^{2,1}_\Delta(\Omega)$. In \cite[Exercise 5.2, page 83]{PonceBook} one finds an example for $\Omega=B_1$ and $N\geq 3$, of a function $u\in W^{2,1}_\Delta(\Omega)$ which is not in $W^{2,1}(\Omega)$. Moreover, in any domain, it is shown that the inequality $\|u\|_{W^{2,1}}\leq C |\Delta u|_{1}$, with $C>0$ a universal constant, may never hold (see Remark~\ref{rem:inequality} for more details).

Following \cite[Section 2]{PRT12}, given $u\in L^1_{loc}(\Omega)$, we define
\begin{align}\label{def:Tnorm}
|\Delta u|_T=\sup \bigg\{\int_{\Omega} u \Delta \varphi : \varphi \in C^\infty_{c}(\Omega),\ |\varphi|_{\infty} \leq 1\bigg\}
\end{align}
and
\begin{align}\label{def:BL}
BL_{0}(\Omega)&:=\big\{u \in W_{0}^{1,1}(\Omega): |\Delta u|_T <\infty\big\},
\end{align}
which is a Banach space when endowed with the norm $|\Delta \cdot |_T$ (see Lemma~\ref{equiv:lem}).
 
\section{An eigenvalue problem}\label{eigen:sec}

Let $p,q$ be such that
\begin{equation}\label{eq:lowerboundbeta}
q\geq 1,\qquad p\geq 1, \qquad  (N-2p)q<Np.
\end{equation}
 
We study first an auxiliary nonlinear eigenvalue problem. Given $p$ and $q$ satisfying \eqref{eq:lowerboundbeta}, let
\begin{align}\label{alphaq}
    \Lambda_{p,q}:=\inf_{u\in W^{2,p}_\Delta(\Omega)\setminus\{0\}}\frac{|\Delta u|_p}{|u|_q}
    =\inf\left\{|\Delta u|_p:u\in W^{2,p}_\Delta(\Omega),\ |u|_q=1\right\}.
\end{align}
In particular, for $p=1$,
\begin{align}\label{alfas}
\Lambda_{1,q}
:=\inf_{u\in W^{2,1}_\Delta(\Omega)\backslash\{0\}}\frac{|\Delta u|_1}{|u|_q}
=\inf_{u\in W^{2,1}_\Delta(\Omega)\backslash\{0\}}\frac{|\Delta u|_T}{|u|_q},
\end{align}
where we used the fact that $|\Delta u|_1=|\Delta u|_T$ for $u\in W^{2,1}_\Delta(\Omega)$ (see Lemma~\ref{lemma:Deltau}). For the following, recall the definition of $1^{**}$ in \eqref{eq:1*1**}.
\begin{lemma}\label{alpha1:lem}
 Let $\Omega$ be convex or of class $C^{1,\gamma}$ for some $\gamma\in (0,1]$. If $q\in[1,1^{**})$, then
\begin{equation}\label{eq:Lambda_{1,p}}
\Lambda_{1,q}= \inf_{u\in BL_0(\Omega)\backslash\{0\}}\frac{|\Delta u|_T}{|u|_q},
\end{equation}
and $\Lambda_{1,q}$ is achieved in $BL_0(\Omega)$; namely, there is $u_1\in BL_0(\Omega)$ such that $|u_1|_q=1$ and $|\Delta u_1|_T=\Lambda_{1,q}$.
\end{lemma}
\begin{proof}
We start by checking \eqref{eq:Lambda_{1,p}}. Since $W^{2,1}_\Delta(\Omega)\subset BL_0(\Omega)$, we have that
\[
\Lambda_{1,q}\geq \inf_{u\in BL_0(\Omega)\backslash\{0\}}\frac{|\Delta u|_T}{|u|_q}.
\]
Conversely, given $u\in BL_0(\Omega)\setminus \{0\}$, by Lemma~\ref{density:lemma} there exists ${(u_n)}_{n\in\N}\subset C^\infty(\Omega)\cap C(\overline \Omega)\cap BL_0(\Omega)$ such that $u_n\to u$ strongly in $W^{1,1}_0(\Omega)$ and $|\Delta u_n|_T\to |\Delta u|_T$. By Lemma~\ref{lemma:density}, $\Delta u_n\in L^1(\Omega)$ and so $u_n\in W^{2,1}_\Delta(\Omega)$ for any $n\in\N$. Moreover, since ${(u_n)}_{n\in\N}$ is bounded in $BL_0(\Omega)$ and $q<1^{**}$, by the compact embedding $BL_0(\Omega)\hookrightarrow L^q(\Omega)$ (see Proposition~\ref{compact:prop}), we have $u_n\to u$ in $L^q(\Omega)$, up to a subsequence. Then,
\[
\frac{|\Delta u|_T}{|u|_q}
=\lim_{n\to\infty}\frac{|\Delta u_n|_T}{|u_n|_q}
=\lim_{n\to\infty}\frac{|\Delta u_n|_1}{|u_n|_q} 
\geq 
\Lambda_{1,q}.
\]
Identity \eqref{eq:Lambda_{1,p}} now follows by taking the infimum in $u\in BL_0(\Omega)$. 

Finally, since
\[
\Lambda_{1,q}=\inf\big\{|\Delta u|_T:\ u\in BL_0(\Omega),\ |u|_q=1\big\},
\]
this infimum is actually a minimum by using again the compact embedding $BL_0(\Omega)\hookrightarrow L^q(\Omega)$ combined with the lower semicontinuity of $|\Delta \cdot|_T$ (Lemma~\ref{lemma:lsc}).
\end{proof}

\begin{remark}
    The previous result was known in the case $q=1$, see \cite{PRT12}.
\end{remark}

\begin{lemma}\label{alphap:lem} Let $\Omega$ be convex or of class $C^{1,\gamma}$, for some $\gamma\in (0,1]$. If $p>1$, then $\Lambda_{p,q}$ is achieved by a function $u\in W^{2,p}_\Delta(\Omega)$ with $|u|_q=1$, and
\[
\Delta (|\Delta u|^{q-2}\Delta u)=\Lambda_{p,q}|u|^{q-2}u \text{ in } \Omega,\qquad u=\Delta u=0 \text{ on } \partial \Omega.
\]
\end{lemma}
\begin{proof}
Under \eqref{eq:lowerboundbeta},  observe that the embedding $W^{2,p}(\Omega)\hookrightarrow L^q(\Omega)$ 
is always compact. Indeed, in the case $2p<N$, the last condition in \eqref{eq:lowerboundbeta}  is equivalent to $q<\frac{N p}{N-2p}$;  in any case, the compactness follows from Lemma~\ref{lemma:Sobolev_embedding} below. 

For $p>1$, the space $W^{2,p}_\Delta(\Omega)$ endowed with the norm $|\Delta \cdot |_p$ is reflexive (see Lemma~\ref{equiv:lem}), hence the infimum $\Lambda_{p,q}$ is always achieved by a function in $W^{2,p}_\Delta(\Omega)$ by the direct method of Calculus of Variations. The remaining statement also standard.
\end{proof}

\subsection{The 1-bilaplacian Lane-Emden equation}\label{sec:1-bilaplacian}
Let $\Omega$ be a bounded domain which is either convex or with $C^{1,\gamma}$ boundary, for some $\gamma\in (0,1]$. Following the ideas and the notations from \cite{PRT12,s86}, we prove the existence of solutions to the (Navier) 1-bilaplacian equation 
\begin{align}\label{1bi}
    \Delta\left(\frac{\Delta u}{|\Delta u|}\right) = |u|^{q-2}u\quad \text{in }\Omega,\qquad u=\frac{\Delta u}{|\Delta u|}=0 \text{ on } \partial \Omega,
\end{align}
for $q\in (1,1^{**})$, making a connection with the nonlinear eigenvalue $\Lambda_{1,q}$. To give a consistent notion of a solution for \eqref{1bi}, we introduce some notation. 

Consider the functional $E:L^{q}(\Omega)\to [0,\infty]$ given by 
\begin{align}\label{E:def}
E(u):=\begin{cases}
    |\Delta u|_T & \text{ if }u\in BL_0(\Omega),\\
    \infty & \text{ if }u\in L^{q}(\Omega)\backslash BL_0(\Omega).
\end{cases}
\end{align}
Observe that $E$ does not always coincide with $|\Delta \cdot |_T$, since $BL_0(\Omega)$ is strictly contained in the set $\{u
\in L^{q}(\Omega):\ |\Delta u|_T<\infty\}$, because the definition \eqref{def:BL} of $BL_0(\Omega)$ directly encodes homogeneous boundary conditions.

\begin{lemma}
 The functional $E$ is convex and lower-semicontinuous in $L^{q}(\Omega)$ for $q\in(1,1^{**})$.
\end{lemma}
\begin{proof}
The convexity follows from the definition of $E$ and the convexity of the functional $|\Delta\cdot|_T$. As for the lower-semicontinuity, let $(u_n)\subset L^{q}(\Omega)$ be a sequence such that $u_n\to u$ in $L^{q}(\Omega)$; in particular, the convergence takes place also in $L^1(\Omega)$. Then, either $\liminf E(u_n)=\infty$ or, on the contrary, we may consider $(u_n') \subset BL_0(\Omega)$, a subsequence of $(u_n)$, such that $\liminf E(u_n)= \liminf |\Delta u'_n|_T<\infty$. Then, by Lemma~\ref{lemma:lsc}, we have $u\in BL_0(\Omega)$ and $E(u)=|\Delta u|_T\leq \liminf |\Delta u_n|_T$. 
\end{proof}
Recall that the subdifferential of $E$ at $u\in L^{q}(\Omega)$ (see, for example, \cite{s86} and \cite[Proposition 4.23]{KS07}) is given by
\begin{align}
    \partial E(u) &:= 
    \bigg\{f\in L^{q'}(\Omega)=L^{\frac{q}{q-1}}(\Omega):E(\varphi)\geq E(u)+\int_\Omega 
    f(\varphi-u)\text{ for any }\varphi\in L^{q}(\Omega)\bigg\},
\end{align}

\begin{prop}\label{prop5p2}Let $u\in BL_0(\Omega)$. Then $f\in \partial E(u)$ if and only if there is $v\in W^{1,1}_0(\Omega)\cap L^\infty(\Omega)$ such that $|v|_\infty\leq 1$, $-\Delta v=f\in L^{\frac{q}{q-1}}(\Omega)$ in the sense of distributions, and $E(u)=-\int_\Omega u\,\Delta v.$
\end{prop}
\begin{proof}
The proof is basically the same as in \cite[Proposition 5.2]{PRT12} (which, in turn, follows closely the proof of \cite[Proposition 4.23]{KS07}).  The only change is that the functional $E$ in \cite[Proposition 5.2]{PRT12} is considered over $L^{1^*}(\Omega)$, but exactly the same proof holds if $E$ is considered instead on $L^{q}(\Omega)$, noting that $q'=\frac{q}{q-1}$. 
\end{proof}

Next, let $G:L^{q}(\Omega)\to \R$ be given by
\begin{align}\label{eq:def_G}
    G(u):=\frac{1}{q}\int_\Omega |u|^{q}.
\end{align}
Since $q>1$, we have that $G$ is convex, of class $C^1$,
\begin{align}
    G'(u)\varphi =
    \int_\Omega|u|^{q-2}u\,\varphi
    \quad\text{for any }\varphi\in L^{q}(\Omega), \qquad \text{ and }\qquad
    \partial G(u) = \big\{|u|^{q-2}u\big\}, \label{Gsub}
\end{align}
see \cite[Proposition 4.23]{KS07}.  Now we define the Euler-Lagrange functional associated with \eqref{1bi} given by
\begin{align}\label{Phi:def}
\Phi:L^{q}(\Omega)\to (-\infty,\infty],\qquad \Phi(u)=E(u)-G(u).    
\end{align}
Observe that its \emph{effective domain}, that is, the set of elements in the domain at which $E$ is finite, is $BL_0(\Omega)$. Following \cite{s86}, we say that $u\in BL_0(\Omega)$ is a critical point of $\Phi$ if
\[
G'(u)\in \partial E(u),
\]
that is, if 
\[
|\Delta f|_T\geq |\Delta u|_T+\int_\Omega |u|^{q-2}u (f-u) \qquad \text{for every }f\in L^\frac{q}{q-1}(\Omega).
\]

\begin{prop}
    Given $u\in BL_0(\Omega)$, we have that $u$ is a critical point of $\Phi$ if and only if
    \begin{equation}\label{v:def}
\begin{cases}\text{there is $v\in W^{1,1}_0(\Omega)\cap L^\infty(\Omega)$ such that} \\
|v|_\infty\leq 1, \quad -\Delta v = |u|^{q-2}u \text{ in the sense of distributions in }\Omega, \quad \text{ and }\quad  -\int_\Omega u\Delta v = |\Delta u|_T.
\end{cases}
\end{equation}
\end{prop}
\begin{proof}
This is a direct consequence of Proposition~\ref{prop5p2} and the definition of critical point.
\end{proof}

\begin{definition} We say that $u\in BL_0(\Omega)$ is a solution to \eqref{1bi} if it is a critical point of $\Phi$ (or, equivalently, if \eqref{v:def} holds true. We say that $u\in BL_0(\Omega)\setminus\{0\}$ is a least energy solution to \eqref{1bi} if $u$ is a solution and
\[
\Phi(u)=c_q:=\min\{\Phi(w):\ w \text{ is a nontrivial solution of } \eqref{1bi}\}.
\]
\end{definition}

\begin{remark}
The function $v$ in \eqref{v:def} gives a consistent meaning to the quotient $-\Delta u/|\Delta u|$.  Observe that $BL_0(\Omega)
 \hookrightarrow L^{q}(\Omega)$ for $q\in[1,1^{**})$, see Proposition~\ref{compact:prop}. Therefore, the integral $\int_\Omega u\,\Delta v$ is well defined.
\end{remark}

\begin{remark}
For each $\lambda>0$, one can define analogously the notion of solution and least energy solution for the 1-Biharmonic equation $\Delta \left(\frac{\Delta u}{|\Delta u|} \right)=\lambda|u|^{q-2}u$ with Navier boundary conditions.
\end{remark}

\begin{prop}\label{prop:existence_of_les} For $q\in (1,1^{**})$, let $u_1\in BL_0(\Omega)\setminus\{0\}$ achieve $\Lambda_{1,q}$  and $|u_1|_q=1$. Then,
\begin{equation}\label{eq:equation_for_constrained_cpt}
\Delta\left(\frac{\Delta u_1}{|\Delta u_1|}\right) = \Lambda_{1,q}|u_1|^{q-2}u_1\quad \text{in }\Omega,\qquad u_1=\frac{\Delta u_1}{|\Delta u_1|}=0 \text{ on } \partial \Omega.
\end{equation}
Moreover,  $u=\Lambda_{1,q}^{\frac{1}{q-1}}u_1$ is a least energy solution of \eqref{1bi} and
\begin{align*}
\min_{v\in L^{q}(\Omega)}\Phi(v)=\frac{q-1}{q}\Lambda_{1,q}^\frac{q}{q-1}
\end{align*}
for $\Phi$ as in \eqref{Phi:def}.
\end{prop}

Before providing the proof of this result, we recall the following version of the Lagrange multiplier rule for nonsmooth convex functionals.

\begin{lemma}[{\cite[Proposition 6.4]{KS07}}]\label{Lagrangemult}
Let $X$ be a Banach space,  $E: X \mapsto (-\infty,\infty]$  be a convex functional and $G: X \mapsto \mathbb{R}$ be convex and continuous. Assume there exists $u\in X$ such that
\begin{enumerate}
\item $E(u)=\min \{E(v):\ v\in X,\ G(v)=1\}$;
\item there exists $\tilde{u} \in X$ such that
\[
E(u+\tilde{u})<E(u),\ G(u+\tilde{u})<1,\ G(u-\tilde{u})<\infty .
\]
\end{enumerate}
Then
\[
\partial G(u) \subset \bigcup_{t \geq 0} t \partial E(u).
\]
\end{lemma}
\begin{proof}[Proof of Proposition~\ref{prop:existence_of_les}]
Observe that
\[
0<\Lambda_{1,q}=\min\{E(u):\ u\in L^{q}(\Omega),\ |u|_q^q=1\}<\infty,
\]
which, by Lemma~\ref{alpha1:lem}, is achieved at some $u_1\in BL_0(\Omega)\subseteq L^{q}(\Omega)$ with $|u_1|_q=1$.  By applying Lemma~\ref{Lagrangemult} with $X=L^q(\Omega)$, $E$ as in \eqref{E:def}, $G$ as in \eqref{eq:def_G}, and $\tilde u =-u_1$ and recalling \eqref{Gsub}, we have 
\[
|u_1|^{q-2}u_1\in\bigcup_{t\geq 0} t \partial E(u_1).
\]
In particular, there are $t\geq 0$ and $f\in \partial E \subset L^{\frac{q}{q-1}}(\Omega)$ such that $|u_1|^{q-2}u_1=tf$.  Since $u_1\neq0$, then $t>0$. Moreover, by Proposition~\ref{prop5p2}, there is $v\in W^{1,1}_0(\Omega)\cap L^\infty(\Omega)$ with $|v|_\infty\leq 1$, $-\Delta v=f$ in distributional sense, and $E(u_1)=-\int_\Omega u_1\Delta v$. Then, we have that $-\Delta v = t^{-1}|u_1|^{q-2}u_1$ and, since $|u_1|_q=1$, 
\[
\Lambda_{1,q}=|\Delta u_1|_T=E(u_1)=-\int_\Omega u_1\,\Delta v=t^{-1}\int_\Omega |u_1|^{q}=t^{-1}.
\]
In particular, $u_1$ is a solution to \eqref{eq:equation_for_constrained_cpt}.

If $u=\Lambda_{1,q}^{\frac{1}{q-1}} u_1$, then
\begin{align*}
-\int_\Omega u\,\Delta v
=-\Lambda_{1,q}^{\frac{1}{q-1}}\int_\Omega u_1\,\Delta v
=\Lambda_{1,q}^{\frac{1}{q-1}} |\Delta u_1|_T
=|\Delta (\Lambda_{1,q}^\frac{1}{q-1}u_1)|_T
= |\Delta u|_T
\end{align*}
and 
\[
-\Delta v=\Lambda_{1,q}|u_1|^{q-2}u_1=|\Lambda_{1,q}^\frac{1}{q-1} u_1|^{q-2}\Lambda_{1,q}^\frac{1}{q-1} u_1=|u|^{q-2}u
\qquad\text{ in }\Omega.
\]
This yields that $u$ is a solution of \eqref{1bi}. A direct calculation shows that
\[
\Phi(u)=\Phi(\Lambda_{1,q}^{\frac{1}{q-1}}u_1)
=\Lambda_{1,q}^\frac{1}{q-1}|\Delta u_1|_T-\frac{1}{q}\Lambda_{1,q}^\frac{q}{q-1}|u_1|_q
=\frac{q-1}{q}\Lambda_{1,q}^\frac{q}{q-1}.
\]
Moreover, if $U\in BL_0(\Omega)$ is a nontrivial solution of \eqref{1bi}, then $|U|_{q}^{q}=|\Delta U|_T$ (by \eqref{v:def}); then $|\Delta U|_T=(\frac{|\Delta U|_T}{|U|_{q}})^\frac{q}{q-1}$ and therefore
\begin{align*}
\Phi(U)
=\left(1-\frac{1}{q}\right)|\Delta U|_T
=\frac{q-1}{q}\left(\frac{|\Delta U|_T}{|U|_{q}}\right)^{\frac{q}{q-1}}
\geq \frac{q-1}{q}\Lambda_{1,q}^{\frac{q}{q-1}}.
\end{align*}
As a consequence, $u$ given by $\Lambda_{1,q}^{\frac{1}{q-1}}u_1$ is a least-energy solution of $\Phi$.
\end{proof}

\begin{remark}
    There are other possible standard characterizations of the least energy level. Indeed, if $q\in (1,2)$, then $c_q$ can be achieved by global minimization:
    \[
    c_q=\min\{\Phi(u):\ u\in L^{q}(\Omega)\}
    \]
    while, for $q\in (2,1^{**})$, it is a Mountain-Pass level
    \[
    c_q=\inf_{\gamma\in \Gamma} \sup_{t\in [0,1]} \Phi(\gamma(t)),
    \]
    where $\Gamma:=\{\gamma\in C([0,1]: \gamma(0)=0,\ \Phi(\gamma(1))<0\}$.
    The proof follows standard arguments combined with results from nonsmooth analysis. It is not hard to check that $\Phi$ satisfies the Palais-Smale condition at any $c\in \R$, namely (see \cite{s86}) that whenever $\left(u_n\right)\subset L^{q}(\Omega)$ is a sequence such that 
\[
\Phi(u_n)\to c\quad \text{ and }\quad E(v) \geq E(u_n) + \int_\Omega |u_n|^{q-2}u_n (v-u_n)+\int_\Omega z_n (v-u_n) \quad \text{ for all } v \in L^q(\Omega),
\]
where $z_n\to 0$ in $L^\frac{q}{q-1}(\Omega)$ as $n\to\infty,$ then $(u_n)$ possesses a convergent subsequence in $L^q(\Omega)$. Then we may use \cite[Theorem 1.7]{s86} when $q\in (1,2)$ (global minimization), and \cite[Theorem 3.2]{s86} when $q\in (2,1^{**})$ (the mountain pass theorem).
\end{remark}

\begin{remark} In \cite[Theorem 1.1]{BP18}, the authors show the existence of least-energy solutions for problems involving the 1-bilaplacian operator and with more general nonlinearities $f(x,s)$  on the right-hand-side (not necessarily homogeneous, superlinear at the origin, and subcritical at infinity).  See also Remark~\ref{rem:BP_comment} below.
\end{remark} 

\subsection{Characterization of solutions}

In \cite[Theorem 1.2]{PRT12}, the authors show that, for any convex domain $\Omega$ or any domain with $C^{1,\gamma}$ boundary, the first eigenvalue of the 1-bilaplacian equation is achieved and it is given by 
\begin{align*}
\Lambda_{1,1}(\Omega)=\inf_{u\in BL_0(\Omega)\backslash \{0\}}\frac{|\Delta u|_T}{|u|_1}=\frac{1}{\Psi_M(\Omega)},
\end{align*}
where $\Psi_M(\Omega)=\max_\Omega \Psi=\Psi(x_M)$ for $x_M\in\Omega$ a global maximum point of $\Psi$,  the torsion function of $\Omega$, namely, the solution of $-\Delta \Psi=1$ in $\Omega$ with $\Psi=0$ on $\partial \Omega$.  Furthermore, the authors in \cite{PRT12} also show that a first eigenfunction is given by  $G_\Omega(\cdot,x_M)$, where $G_\Omega$ is the Green's function for the Dirichlet Laplacian in $\Omega$.

In this section, we study the equivalent result in the nonlinear case $\Lambda_{1,q}$ for $q\in[1,1^{**})$.  For this, we need the following remark. Let
\begin{align}\label{hdef}
h:\overline{\Omega}\to \R\quad \text{ be given by }\quad h(x)=|G_\Omega(\cdot,x)|_q^q.
\end{align}
Then $h$ is well defined (because, by the maximum principle, $0\leq G_\Omega(x,y)\leq C|x-y|^{2-N}$ for some $C>0$ and because $q\in [1,1^{**})$). Furthermore, $h$ is continuous in $\overline{\Omega}$. Indeed, let $x\in \overline{\Omega}$, $(x_n)\subset \Omega$ be a sequence such that $x_n\to x$ as $n\to\infty$, and let $B$ be a large ball so that $\Omega-x_n\subset B$ for all $n$; then, by a change of variables,
\[
|G_\Omega(x_n,\cdot)|_q^q=\int_B |G_\Omega(x_n,y+x_n)|^q\chi_{\Omega}(y+x_n)\, dy.
\]
Since $|G_\Omega(x_n,y+x_n)|^q\chi_{\Omega}(y+x_n)\to |G_\Omega(x,y+x)|^q\chi_{\Omega}(y+x)$ as $n\to\infty$ for a.e. $y\in B$ and
\begin{align*}
|G_\Omega(x_n,y+x_n)|^q\chi_{\Omega}(y+x_n)\leq \frac{1}{|y|^{q(N-2)}}\in L^1(B),    
\end{align*}
then, by dominated convergence,
\[
h(x_n)=|G_\Omega(x_n,\cdot)|_q^q\to \int_B|G_\Omega(x,y+x)|^q\chi_{\Omega}(y+x)\, dy=|G_\Omega(x,\cdot)|_q^q=h(x).
\]
Finally, note also that $h(x)=0$ for all $x\in \partial \Omega$ (because $G_\Omega(x,y)=0$ for all $x\in\partial \Omega$ and $y\in\Omega$). These facts guarantee the existence of $x_M\in\Omega$ such that $h(x_M)=\max_{\overline{\Omega}} h.$

\begin{prop}\label{prop1.7-part1}
Let $\Omega$ be convex or of class $C^{1,\gamma}$ for some $\gamma\in (0,1]$ and $q\in [1,1^{**})$.  Then,
\begin{align*}
\Lambda_{1,q}(\Omega)=\frac{1}{|G_\Omega(\cdot,x_M)|_q},
\end{align*}
where $x_M\in \Omega$ is a maximum of the function $h$ introduced in \eqref{hdef}.
Moreover, the function
\begin{align}\label{exp:for2}
u_1(x)=\frac{G_\Omega(x,x_M)}{|G_\Omega(\cdot,x_M)|_q}\qquad \text{ for }x\in \Omega
\end{align}
achieves $\Lambda_{1,q}$
\end{prop}
\begin{proof}
Note that
\begin{align*}
\Lambda_{1,q}(\Omega)\leq \frac{|\Delta G_\Omega(x_M,\cdot)|_T}{|G_\Omega(x_M,\cdot)|_q}=\frac{1}{|G_\Omega(x_M,\cdot)|_q},
\end{align*}
because $G_\Omega(x_M,\cdot)\in BL_0(\Omega)$ and
\begin{align*}
    |\Delta G_\Omega(x_M,\cdot)|_T
    &=
    \sup\bigg\{\int_\Omega G_\Omega(x_M,\cdot)\Delta\varphi:\varphi\in C^2_c(\Omega),\ |\varphi|_\infty\leq 1\bigg\}\\
    &=
    \sup\bigg\{-\varphi(x_M):\varphi\in C^2_c(\Omega),\ |\varphi|_\infty\leq 1\bigg\}=1.
    \end{align*}
On the other hand, let $u_1\in BL_0(\Omega)$ be a minimizer for $\Lambda_{1,q}(\Omega)$ such that $|u_1|_q=1$. Then $-\Delta u_1=\mu$ for some Radon measure $\mu$ and, in particular, $u_1(x)=\int_\Omega G_\Omega(x,y)\, d\mu(y)$ for $x\in \Omega$ (see Remark~\ref{nico:rmk}). Moreover, by Jensen's inequality, Fubini's theorem, and \eqref{xMdef},
\begin{align}
1&=|u_1|^q_q
=\int_\Omega \bigg|\int_\Omega G_\Omega(x,y)\, d\mu(y) \bigg|^q dx
\leq |\mu|(\Omega)^{q-1}\int_\Omega \int_\Omega G_\Omega(x,y)^q\, d|\mu|(y)\, dx\notag\\
&=|\mu|(\Omega)^{q-1}\int_\Omega \int_\Omega G_\Omega(x,y)^q\, dx\, d|\mu|(y)
\leq |\mu|(\Omega)^{q-1}|G_\Omega(\cdot,x_M)|_q^q\int_\Omega \, d|\mu|(y)\notag\\
&=|\mu|(\Omega)^q|G_\Omega(\cdot,x_M)|_q^q,\label{eq1}
\end{align}
which yields, by \eqref{measures_equiv}, that $\Lambda_{1,q}=|\Delta u_1|_T=|\mu|(\Omega)\geq \frac{1}{|G_\Omega(\cdot,x_M)|_q}$.
\end{proof}

\begin{remark}
    Note that, if $q=1$, then $h(x)=|G(x,\cdot)|_{1}=\int_{\Omega}G_\Omega(x,y)\, dy=\Psi(x)$, where $\Psi$ is the Dirichlet torsion function of $\Omega.$  Therefore, \cite[Theorem 1.2]{PRT12} is recovered from Theorem~\ref{char:thm}, when $q=1$.
\end{remark}

\begin{coro}\label{explicit:cor}
Let $\Omega$ be convex or of class $C^{1,\gamma}$, for some $\gamma\in (0,1]$. Let $x_M\in \Omega$ be such that
\begin{align*}
    |G_\Omega(\cdot,x_M)|_q=\max_{x\in \overline{\Omega}}|G_\Omega(\cdot,x)|_q.
\end{align*}
 Then, the function 
\begin{align*}
    u(x)=|G_\Omega(\cdot,x_M)|_q^{-\frac{q}{q-1}} G_\Omega(x,x_M)\qquad\text{for }x\in \Omega
\end{align*}
 is a least-energy solution of \eqref{1bi}.
\end{coro}
\begin{proof}
The claim follows from Proposition~\ref{prop:existence_of_les} and Theorem~\ref{char:thm}.
\end{proof}

The next two results extend \cite[Proposition 5.6]{PRT12} and \cite[Proposition 6.6]{PRT12} to our setting, with very similar proofs.

\begin{prop}\label{pos:prop}
    Let $u\in BL_0(\Omega)$ be a minimizer for $\Lambda_{1,q}(\Omega).$  Then, up to a sign,
    $\mu:=-\Delta u$ is a positive Radon measure and $u$ is positive in $\Omega$.
\end{prop}
\begin{proof}
The proof is exactly the same as in \cite[Proposition 5.6]{PRT12} using $L^q$ norms instead of $L^1$ norms. 
\end{proof}

\begin{prop}\label{prop:uniqueness}
Let $u_1$ be a minimizer for $\Lambda_{1,q}(\Omega)$ such that $|u_1|_q=1$ and let $h$ be as in \eqref{hdef}. If $h$ admits only one maximum point $x_M\in\Omega$, then, $u_1$ is given by \eqref{exp:for2} and it is the unique (up to sign) minimizer of $\Lambda_{1,q}(\Omega)$ such that $|u_1|_q=1$.
\end{prop}
\begin{proof}
Let $u_1\in BL_0(\Omega)$ be a minimizer for $\Lambda_{1,q}(\Omega)$. Then, without loss of generality, we may assume that $\mu=-\Delta u_1$ is a positive Radon measure (by Proposition~\ref{pos:prop}). By the definition of $h$ and $x_M$, we have that \eqref{eq1} holds with equalities; in particular,
\begin{align}
\int_\Omega h(y)\, d\mu(y)=\int_\Omega \int_\Omega |G_\Omega(x,y)|^q\, dx\, d\mu(y)
=|G_\Omega(\cdot,x_M)|_q^q\int_\Omega \, d\mu(y)
=\int_\Omega h(x_M) \, d\mu(y),
\end{align}
with $h$ as in \eqref{hdef}.  Since $\mu$ is a positive Radon measure and $h(x_M)-h(y)\geq 0$ for all $y\in \Omega,$ we deduce that $\mu$ has support in the set $\{y\in\Omega:h(x_M)=h(y)\}=\{x_M\}$.  This implies that $\mu=\mu(\Omega)\delta_{x_M}$ and the claim follows.
\end{proof}

\begin{proof}[Proof of Theorem~\ref{char:thm}]
    This is now a consequence of Propositions~\ref{prop1.7-part1},~\ref{pos:prop}, and~\ref{prop:uniqueness}.
\end{proof}

\begin{remark}\label{rmk:q} 
In Section~\ref{sec:limit:ball} we show that, if $\Omega$ is a ball, then the point $x_M\in \Omega$ given by \eqref{xMdef} can be taken as the origin $x_M=0$. However, in more general domains, we conjecture that the point $x_M=x_M(q)$ may vary depending on $q$.  To see this, here we only provide a \emph{heuristic argument}\footnote{We thank Sven Jarohs and Tobias Weth for helpful discussions in this regard.}. Consider a ``spinning top" domain $\Omega$ given by 
\begin{align*}
\Omega=\{(x,y,z)\in \R^3 \::\: |(x,y)|<1, \, z\in(|(x,y)|,2)\},    
\end{align*}
see Figure~\ref{Rfig}.  By symmetry and convexity, it is expected that the point $x_M$ lies on the set $\{(x,y,z)\in\Omega\::\: x=y=0\}$ (the dotted vertical line in Figure~\ref{Rfig}).  For $q=1$, the point $x_M(1)$ should be close to the center of mass of $\Omega$. As $q$ increases, this enhances the singularity of the Green function and penalizes more the points that are close to the boundary (note that $G(x,\cdot)\to 0$ as $\dist(x,\partial \Omega)\to 0$ and $|G(x,\cdot)|^q$ goes faster to zero for $q$ large).  As a consequence, the point $x_M(q)$ should move towards the insphere center, namely, the center of the largest sphere contained in $\Omega$ (represented by the dotted circle in Figure~\ref{Rfig}).  This phenomenon can already be seen, for example, with the fundamental solution in dimension 3.  Indeed, for $N=3$ and $q\in(0,3)$, consider the function $h:(0,2)\to (0,\infty)$ given by
\begin{align}\label{eq:Lq_of_fundamental}
    h_q(t)
    :=\int_\Omega |(0,0,t)-(x,y,z)|^{-q}\, dx\,dy\,dz
    =2\pi\int_0^1\rho\int_{\rho-t}^{2-t} |\rho^2+z^2|^{-q}\, dz\, d\rho.
\end{align}
Since we are interested in the (unique) maximum point of $h_q$, we differentiate and obtain that
\begin{align*}
h_q'(t)=2\pi\int_0^1\rho(|\rho^2+(\rho-t)^2|^{-q}-|\rho^2+(2-t)^2|^{-q})\, d\rho.
\end{align*}
These integrals can be computed explicitly for $q=1$ and $q=2$, and the corresponding roots can be found numerically. In particular, the points $y_M(1)\sim 1.28$ and $y_M(2)\sim1.27$ maximize $h_1$ and $h_2$ respectively.

\setlength{\unitlength}{1cm}
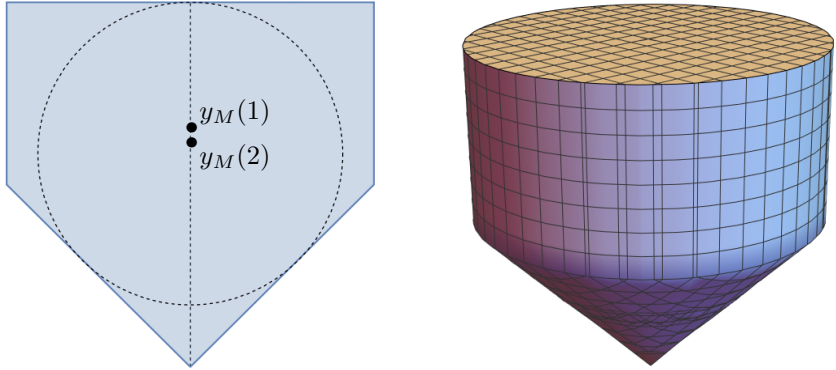
\begin{figure}[ht]
\begin{center}
\begin{picture}(11,5)(0,0)
\put(0,0){\includegraphics[width=5cm]{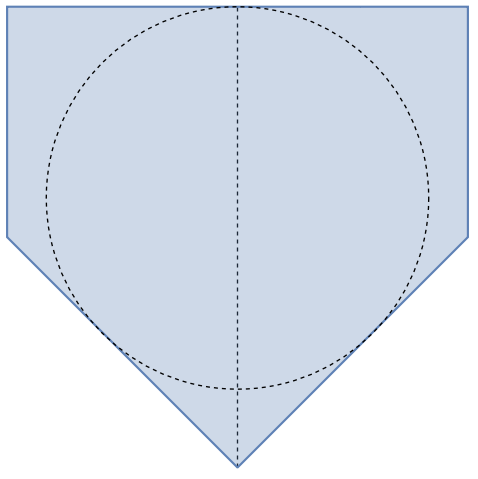}}
\put(6,0){\includegraphics[width=5cm]{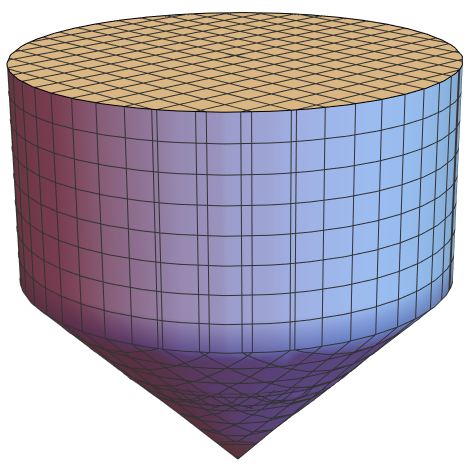}}
\put(2.41,3.2){$\bullet$}
\put(2.6,3.4){$y_M(1)$}
\put(2.41,3){$\bullet$}
\put(2.6,2.8){$y_{M}(2)$}
\end{picture}
\end{center}
\caption{On the right, the ``spinning top'' domain $\Omega$ introduced in Remark~\ref{rmk:q}. On the left, a transversal cut of $\Omega$ showing the maximizing points $y_{M}(1)$ (which is closer to the center of mass) and $y_{M}(2)$ (which is closer to the center of the largest sphere contained in $\Omega$), where for each $q$ and $N=3$, $y_M(q)$ denotes the maximum point of the $L^q$ norm of the fundamental solution, given by \eqref{eq:Lq_of_fundamental}.}
\label{Rfig}
\end{figure}
\end{remark}

\subsection{Asymptotics for \texorpdfstring{$p$}{p} and \texorpdfstring{$q$}{q}}

\begin{remark}\label{rem:beta=1}
Note that, when $p=1$, then \eqref{eq:lowerboundbeta} implies that
\[
q\in[1,1^{**})\quad\text{or, equivalently,}\quad q-1\in[0,1^{**}-1),
\] 
(see \eqref{eq:1*1**}).
\end{remark}
 Recall the definition of $\Lambda_{p,q}$ given in \eqref{alphaq}. The following result provides some uniform bounds for the minimizers.

\begin{lemma}\label{bounds:lemma}
Let $(p_n,q_n)$ and $(p,q)$  satisfy \eqref{eq:lowerboundbeta} for every $n\in \N$, and let $p_n\to p$ and $q_n\to q$ as $n\to\infty$. For any $n\in\N$, consider a minimizer $u_n$ for $\Lambda_{p_n,q_n}$ such that $|u_n|_{q_n}=1$. Then there exists $C>0$ such that, for every $n\in\N$,
\[
|\Delta u_n|_{p_n}\leq C, \quad\text{whenever }p_n>1,
\qquad \text{or} \qquad 
|\Delta u_n|_T\leq C, \quad\text{when } p_n=1.
\]
\end{lemma}
\begin{proof}
For every $n\in\N$, fix $\varphi\in C^\infty_c(\Omega)\setminus \{0\}$, which, in particular, lies in $W^ {2,p_n}_\Delta(\Omega)$, as well as in $BL_0(\Omega)$. 

Assume that $p_n>1$ for every $n\in\N$. Using that $\Lambda_{p_n,q_n}$ is an infimum and dominated convergence,
\begin{align*}
|\Delta u_n|_{p_n} \leq \frac{|\Delta \varphi|_{p_n}}{|\varphi|_{q_n}} \to \frac{|\Delta \varphi|_p}{|\varphi|_q}
\qquad\text{as }n\to\infty,
\end{align*} 
so that $|\Delta u_n|_{p_n} $ is uniformly bounded.

If, on the other hand, $\beta_n=1$ for infinitely many $n\in\N$, then
\begin{align*}
|\Delta u_n|_T \leq \frac{|\Delta \varphi|_{T}}{|\varphi|_{q_n}} \to \frac{|\Delta \varphi|_{1}}{|\varphi|_q}
\qquad\text{as }n\to\infty,
\end{align*} 
so that $|\Delta u_n|_T$ is also uniformly bounded.
\end{proof}

Next, we show the convergence of minimizers.

\begin{prop}\label{c1:prop2} 
The map
\[
(p,q)\mapsto \Lambda_{p,q}
\]
is continuous for $(p,q)$ satisfying \eqref{eq:lowerboundbeta}.

Let $(p_n,q_n)$ and $(p,q)$  satisfy \eqref{eq:lowerboundbeta} for every $n\in \N$, and let $p_n\to 1$ and $q_n\to q$ as $n\to\infty$. Let $u_n$ be the minimizer for $\Lambda_{p_n,q_n}$ with $|u_n|_{q_n}=1$.
There is a minimizer $u$ of $\Lambda_{p,q}$ such that:
\begin{enumerate}
    \item $u \in  BL_0(\Omega)\cap W^{1,r}_0(\Omega)$ for every $r\in [1,1^*)$, with $|u|_q=1$;   
    \item  $u_n\to u$ in $W^{1,r}_0(\Omega)$ for every $r\in [1,1^*)$ and $|\Delta u_n|_T\to |\Delta u|_T$, up to a subsequence.
\end{enumerate}
\end{prop}
\begin{proof}
We first focus on a situation where $p_n>1$ for all $n\in\N$, so that minimizers of $\Lambda_{p_n,q_n}$ belong to $W^ {2,p_n}_\Delta(\Omega)$. Recall also that $q<1^{**}$ in this case (Remark \eqref{rem:beta=1}).  Then
\begin{equation}\label{aux:conv_pn+1}
|\Delta u_n|_1 \leq  |\Omega|^\frac{p_n-1}{p_n} |\Delta u_n|_{p_n}=|\Omega|^\frac{p_n-1}{p_n}\Lambda_{p_n,q_n}.
\end{equation}
Therefore, by Lemma~\ref{bounds:lemma} and Lemma~\ref{equiv:lem}-2, we have that ${(u_n)}_{n\in\N}$ is uniformly bounded in $W^{2,1}_\Delta(\Omega)$. By Proposition~\ref{compact:prop}  (see also Remark~\ref{rem:compactness}), there is $u\in W^{1,r}_0(\Omega)$ such that, up to extracting a subsequence, $u_n\to u$  strongly in $W^{1,r}_0(\Omega)$ for every $r\in[1,1^*)$ and strongly in  $L^{t}(\Omega)$ for $t\in [1,1^{**})$. Moreover, by Lemma~\ref{lemma:lsc}, $u\in BL_0(\Omega)$ and
\begin{equation}\label{eq:lowersemicont_aux}
|\Delta u|_T\leq \liminf_{n\to\infty} |\Delta u_n|_T=\liminf_{n\to\infty} |\Delta u_n|_1.
\end{equation}
Let $\bar t=\frac{q+1^{**}}{2}$. For large $n$, we have $q_n<\bar t$ and $u_n\to u$ strongly in $L^{\bar t}(\Omega)$. Hence, up to a sequence, $u_n\to u$ a.e., and there exists $h\in L^{\bar t}(\Omega)$ such that $|u_n|\leq h$ a.e. in $\Omega$ (see, for example, \cite[Lemma A.1]{W97}). By dominated convergence, we have that
\begin{equation}\label{eq_normLq=1}
1=|u_n|_{q_n}\to |u|_{q}.
\end{equation}
Fix $\eps>0$. By the definition of $\Lambda_{1,q}$ (see \eqref{alfas}), there is $v_1\in W^{2,1}_\Delta(\Omega)\setminus \{0\}$ such that
    \[
    \frac{|\Delta v_1|_T}{|v_1|_{q}}\leq  \Lambda_{1,q}+\eps.
    \]
     Since $v_1\in W^{2,1}_\Delta(\Omega)$, we have $|\Delta v_1|_T=|\Delta v_1|_1$, and (by Lemma~\ref{lemma:densityC^2}) there is $(v_k)\subset C^2(\overline\Omega)\subset W^{2,p_n}(\Omega)$ for every $n$ such that $\Delta v_k \to \Delta v_1$ in $L^1(\Omega)$ and (again by compact embeddings, as $q<1^{**}$), $v_k\to v$ in $L^{q}(\Omega)$.  Then
     \begin{align*}
     \Lambda_{1,q}+\eps \geq & \frac{|\Delta v_1|_1}{|v_1|_{q}}
     =\lim_{k\to\infty} \frac{|\Delta v_k |_1}{|v_k|_{q}} =\lim_{k\to\infty} \lim_{n\to \infty} \frac{|\Delta v_k |_{p_n}}{|v_k|_{q_n}} \geq \lim_{k\to\infty} \lim_{n\to \infty}\Lambda_{p_n,q_n}\\
     &=\lim_{n\to \infty}  \Lambda_{p_n,q_n} = \lim_{n\to \infty}|\Delta u_n |_{p_n} \geq  \lim_{n\to \infty}|\Omega|^{-\frac{p_n-1}{p_n}} | \Delta u_n |_1\\
     &=\lim_{n\to \infty}|\Delta u_n|_1\geq |\Delta u|_T\geq  \Lambda_{1,q},
     \end{align*}
where we have used Lebesgue dominated convergence, \eqref{aux:conv_pn+1}, H\"older's inequality, that $u_n$ is a minimizer for $\Lambda_{p_n,q_n}$,  \eqref{eq:lowersemicont_aux} and \eqref{eq_normLq=1}. Summarizing, $|u|_{q}=1$,
\[
\Lambda_{1,q}\leq |\Delta u_1|_T  \leq \lim_{n\to \infty} |\Delta u_n|_1= \lim_{n\to \infty}  \Lambda_{p_n,q_n}\leq  \Lambda_{1,q}+\eps \qquad \text{ for every } \eps>0.
\]
By letting $\eps\to 0$, equality holds true and the conclusion of the lemma follows for case 1 under the situation $p_n>1$.

Assume now that $p=p_n=1$ for every $n$. Then the minimizer $u_n$ of $\Lambda_{p_n,q_n}=\Lambda_{1,q_n}$ belongs to $BL_0(\Omega)$ and not necessarily to $W^{2,1}_\Delta(\Omega)$. By Lemma~\ref{bounds:lemma}, we have
\[
|\Delta u_n|_T\leq C
\]
and we can repeat almost word by word the proof of the previous case, simply replacing $|\Delta \cdot |_{p_n}$ with $|\Delta \cdot|_T$.
\end{proof}

\begin{proof}[Proof of Theorem~\ref{thm:nonlineareigenvalue_pq}]
Parts (a) and (b) are a direct consequence of Lemmas~\ref{alpha1:lem},~\ref{alphap:lem} and Proposition~\ref{prop:existence_of_les}, while the last paragraph is a consequence of Proposition~\ref{c1:prop2}.
\end{proof}

\begin{remark}\label{rmk:convergencePRT} In the case $p=1$, since $W^{2,1}(\Omega)$ is not reflexive, we cannot say that $u_n\rightharpoonup u$ weakly in $W^{2,1}(\Omega)$. This kind of convergence plays a key role in the asymptotic analysis of solutions. To overcome this difficulty, we use the space $BL_0(\Omega)$, as in \cite{PRT12}. In this paper, the authors studied in detail the quantity $\Lambda_{1,1}$, which is associated with an  eigenvalue problem for the $1$-biharmonic operator:
\begin{equation}\label{eq:1biharmoniclinea}
\Delta \left(\frac{\Delta u}{|\Delta u|}\right)=\Lambda_{1,1}\frac{u}{|u|} \text{ in } \Omega,\qquad u=\frac{\Delta u}{|\Delta u|}=0 \text{ in }\partial \Omega.
\end{equation}
See \cite{PRT12} for the notion of solution, which does not fall in the framework of our Section~\ref{sec:1-bilaplacian}, since the function $G_1(u)=|u|_1$ is not of class $C^1$. As a particular case of our study, we show that, as $p,q\to 1$ ($p\geq 1$), $\Lambda_{p,q}\to \Lambda_{1,1}$, and that solutions of 
\[
\Delta \left(|\Delta u|^{p-2}\Delta u\right)=\Lambda_{p,q}|u|^{q-2}u \text{ in } \Omega,\qquad u=\frac{\Delta u}{|\Delta u|}=0 \text{ in }\partial \Omega,
\]
 (which minimize $\Lambda_{p,q}$), converge to solutions of \eqref{eq:1biharmoniclinea} (which minimize $\Lambda_{1,1}$).
\end{remark}

\section{Hamiltonian systems and limit profiles}\label{sec:Hamiltonian}

Let $\Omega\subset \R^N$ be a bounded domain of class $C^{2,\gamma}$. For $\a,\b$ satisfying \eqref{subcritical}, by \cite[Theorem 4.5]{BMT14}, there is a positive least-energy strong solution $(U_\beta,V_\beta)$ of the system 
\begin{align}\label{hsystem}
 -\Delta U_\b = |V_\b|^{\b-1}V_\b
 \quad \text{ in }\Omega,\qquad
 -\Delta V_\b = |U_\b|^{\a-1}U_\b
 \quad \text{ in }\Omega,\qquad 
 U_\b=V_\b= 0\quad \text{ on }\partial \Omega.
 \end{align}
More precisely, one shows that there exists $U_\beta\in W^{2,\frac{\beta}{\beta+1}}(\Omega)\cap W^{1,\frac{\beta+1}{\beta}}_0(\Omega)$ such that
\[
J_{\alpha,\beta}(U_\beta)=\inf\{J_{\a,\b}(U):\ U \text{ is a nontrivial critical point of } J_{\alpha,\beta}\},
\]
where $J_{\alpha,\beta}$ is as in \eqref{J:intro}. In particular, 
\[
\Delta(|\Delta U_\b|^{\frac{1}{\b}-1}\Delta U_\b)=|U_\b|^{\a-1}U_\b\quad \text{ in }\Omega,\qquad U_\b=\Delta U_\b=0\quad \text{ on }\partial \Omega
\]
and, for $V_\beta:=-|\Delta U_\beta|^{\frac{1}{\beta}-1}\Delta U_\beta$, we have $V_\beta\in W^{2,\frac{\a}{\a+1}}(\Omega)\cap W^{1,\frac{\a+1}{\beta}}_0(\Omega)$ and $(U_\beta,V_\beta)$ is a (strong) solution of \eqref{hsystem:intro}. Moreover, $U_\beta\cdot V_\beta>0$ in $\Omega$, and a standard bootstrap argument yields that
\begin{align}\label{class:sol}
(U_\b,V_\b)\in C^{2,\sigma_1}(\Omega)\times C^{2,\sigma_2}(\Omega)    
\end{align}
for some $\sigma_1,\sigma_2\in(0,1)$. For the details, see \cite[Section 4]{BMT14} and references therein.

\begin{theorem}\label{Uq:conv:thm}
 There is a (least-energy) solution $U_\infty\in BL_0(\Omega)$ of \eqref{1bi} such that, up to a subsequence, as $\beta\to \infty$:
\begin{align}
U_\b\to U_\infty\quad \text{in $W^{1,r}_0(\Omega)$ for every $r\in[1,1^*)$},\quad \text{ and } \quad |\Delta U_\beta|_1\to |\Delta U_\infty|_T.
\end{align}
Moreover, $U_\infty=\Lambda_{1,\a+1}^\frac{1}{\a} u_1$, where $u_1$ is a minimizer for $\Lambda_{1,\a+1}$, given in \eqref{alfas}, with $|u_1|_{\a+1}=1$.
\end{theorem}
\begin{proof}
 By \cite[Lemma 4.8]{BMT14},
\begin{align}\label{eq:scaling}
u_\b=\Lambda_{1+\frac1\b,\a+1}^{-\frac{\b}{\a\b-1}}U_\b
\end{align}
is a minimizer for $\Lambda_{1+\frac1\b,\a+1}$ with $|u_\beta|_{\a+1}=1$. Then, by Proposition~\ref{c1:prop2},
\begin{align*}
U_\b
=\Lambda_{1+\frac1\beta,\a+1}^{\frac{\b}{\a\b-1}}u_\beta
\to \Lambda_{1,\a+1}^\frac{1}{\a}u_1=U_\infty\quad \text{ in $W^{1,r}_0(\Omega)$ as $\b\to\infty$}
\end{align*}
and
\[
|\Delta U_\beta|_1=\Lambda_{1+\frac1\beta,\a+1}^{\frac{\b}{\a\b-1}}|\Delta u_\beta|_1\to  \Lambda_{1,\a+1}^\frac{1}{\a}|u_1|_T =|\Delta U_\infty|_T.
\]

By Proposition~\ref{prop:existence_of_les}, $U_\infty$ is a least-energy solution of \eqref{1bi}.
\end{proof}

Theorem~\ref{Uq:conv:thm} also yields the convergence for the component $V_\b$.
\begin{coro}\label{Vq:conv:coro}
There is $V_\infty\in W^{2,1+\frac1\a}(\Omega)\cap W_0^{1,1+\frac1\a}(\Omega)$ such that
\begin{align*}
    V_\b\rightharpoonup V_\infty\qquad \text{weakly in }W^{2,1+\frac{1}{\alpha}}(\Omega).
\end{align*}

In particular:
\begin{enumerate}
\item if $\alpha\in (0,1^*-1)$, then $V_\infty\in C^{1, \sigma(\alpha)}(\overline \Omega)$ and 
\[
V_\beta\to V_\infty \text{ in } C^{1,\sigma}(\overline \Omega),
\text{ for every  } 0<\sigma<\sigma(\alpha):=1-\frac{N\alpha}{\alpha+1};
\]
\item if $\alpha=1^*-1=\frac{1}{N-1}$, then $V_\infty\in C^{0, \gamma}(\overline \Omega)$ and
\[
V_\beta\to V_\infty\quad  \text{ strongly in } W^{1,\eta}_0(\Omega) \text{ and } C^{0,\sigma}(\overline \Omega), \text{ for every $\eta\geq 1$ and $0<\sigma<1$;}
\]
\item if $\alpha\in (1^*-1,1^{**}-1)$, then  $V_\infty\in W^{1,\frac{N(\alpha+1)}{\alpha(N-1)-1}}_0(\Omega) \cap C^{0, \sigma(\alpha)+1}(\overline \Omega)$ and 
\begin{multline*}
V_\beta\to V_\infty\quad  \text{ strongly in } W^{1,\eta}_0(\Omega) \text{ and } C^{0,\sigma}(\overline \Omega), \\
\text{ for every } 1\leq \eta<\frac{N(\alpha+1)}{\alpha(N-1)-1}\text{ and } 0<\sigma<\sigma
(\a)+1.
\end{multline*}
\end{enumerate}
\end{coro}
\begin{proof}
Since $(U_\b,V_\b)$ is a least-energy solution of \eqref{hsystem}, there is a constant $C_1>0$ independent of $\b$ such that $|U_\b|_{\a+1}<C_1$. Indeed, recalling \eqref{eq:scaling}, we have
\[
|U_\beta|_{\a+1}=\Lambda_{1+\frac{1}{\b},\a+1}^\frac{\b}{\a\b-1}|u_\b|_{\a+1}=\Lambda_{1+\frac{1}{\b},\a+1}^\frac{\b}{\a\b-1},
\]
which is bounded for $\a$ fixed and large $\b$, by Proposition~\ref{c1:prop2}.

Then, by elliptic regularity, there is $C_2>0$ independent of $q$ such that 
\[
\|V_\b\|_{W^{2,1+\frac1\a}}\leq C_2 \big||U_\b|^\a\big|_{1+\frac1\a}=C_2|U_\b|_{\a+1}^{\a}<C_2\,C_1^{p}. 
\]
Then, there is $V_\infty\in W^{2,1+\frac1\a}(\Omega)\cap W_0^{1,1+\frac1\a}(\Omega)$ with $V_\b\weakto V_\infty$ weakly in $W^{2,1+\frac1\a}(\Omega)$.
The other statements follow from Sobolev embeddings (Lemma~\ref{lemma:Sobolev_embedding}), observing that $2\frac{\alpha+1}{\alpha}>N$ if and only if $\alpha<1^{**}-1$, that
\[
0< 1-\frac{N\alpha}{\alpha+1}<1 \iff \alpha < 1^*-1, \quad \text{ and } \quad \left(\frac{\alpha+1}{\a}\right)^*=\frac{N(\alpha+1)}{\alpha(N-1)-1}.
\]
\end{proof}

We are ready to show Theorem~\ref{Uq:conv:thm:intro}.

\begin{proof}[Proof of Theorem~\ref{Uq:conv:thm:intro}]
The theorem follows from Theorem~\ref{Uq:conv:thm} and Corollary~\ref{Vq:conv:coro}.  Note that, since $U_\b\to U_\infty$ in $W^{1,r}_0(\Omega)$ for all $r\in[1,1^*)$, then $U^\a_\b\to U^\a_\infty$ in $L^t(\Omega)$ for all $t\in [1,1^{**}/\a)$, therefore, for a.e. $x\in \Omega$, 
\begin{align*}
   V_\infty(x)
   =\lim_{\b\to\infty} V_\b(x) 
   =\lim_{\b\to\infty} \int_{\Omega} G_{\Omega}(x,y)\,U_\b^\a(y)\;dy
   =\int_{\Omega} G_{\Omega}(x,y)\,U_\infty^\alpha(y)\;dy,
\end{align*}
where $G_{\Omega}$ denotes the Green function of the Dirichlet Laplacian in $\Omega$ and where we have used that it lies in the dual of $L^t(\Omega)$. \end{proof}

\section{The case of the ball}\label{ball:sec}

In the case of the ball some explicit formulas can be obtained.  We collect some auxiliary lemmas first. Let, for $x\in \overline{B_1}\backslash \{0\}$,
$N\geq 3$. 
\begin{align}\label{F:def}
    G_{B_1}(x,0)=c_N(|x|^{2-N}-1),\qquad c_N=\frac{1}{(N-2)|\partial B_1|}=\frac{\Gamma(\frac{N}{2})}{2(N-2)\pi^{N/2}}.\end{align}
In particular, 
\begin{align}\label{FunSol}
-\Delta G_{B_1}(\cdot ,0) = \delta_0\quad \text{in }B_1,
\qquad G_{B_1}(\cdot ,0)=0\quad\text{on }\partial B_1.     
\end{align}

\begin{lemma}\label{Fnorm:lem}
Let $\Omega=B_1$. Then
\begin{align*}
    |\Delta G_{B_1}(\cdot ,0)|_T=1\quad \text{ and }\quad |G_{B_1}(\cdot ,0)|_q&=c_N\left(\frac{\pi^{N/2}}{\Gamma(\frac{N}2+1)}
    \frac{\Gamma(\frac{N}{N-2}-q)\,\Gamma(q+1)}{\Gamma(\frac{N}{N-2})}\right)^\frac{1}{q}.
\end{align*}
\end{lemma}
\begin{proof}
    Passing to spherical coordinates and applying the change of variables $\tau=t^{N-2}$, 
    \begin{align*}
    |G_{B_1}(\cdot ,0)|_q^q &=
    c_N^q\int_B\big(|x|^{2-N}-1\big)^q\;dx=
    c_N^q|\partial B_1|\int_0^1\big(t^{2-N}-1\big)^qt^{N-1}\;dt \\
    &=
    c_N^q|\partial B_1|
    \int_0^1t^{N-1-(N-2)q}\big(1-t^{N-2}\big)^q\;dt=
    \frac{c_N^q|\partial B_1|}{N-2}
    \int_0^1\tau^{\frac{N-1}{N-2}-q}\big(1-\tau\big)^{q}\tau^{-\frac{N-3}{N-2}}\;d\tau\\
    &=
    \frac{c_N^q|\partial B_1|}{N-2}
    \int_0^1\tau^{\frac{2}{N-2}-q}\big(1-\tau\big)^q\;d\tau=\mathrm{B}\left(\frac{2}{N-2}-q+1,q+1\right)\\
    &=\frac{c_N^q|\partial B_1|}{N-2}\frac{\Gamma(q+1)\Gamma\left(\frac{2}{N-2}-q+1\right)}{\Gamma
   \left(\frac{2(N-1)}{N-2}\right)}
   =c_N^q\frac{\pi^{N/2}\Gamma(q+1)\Gamma\left(\frac{2}{N-2}-q+1\right)}{\Gamma
   \left(\frac{N}{2}+1\right) \Gamma \left(\frac{N}{N-2}\right)},
    \end{align*}
    where we used the definition of the Beta function $\mathrm{B}(\cdot, \cdot)$ and its relation with the Gamma function $\Gamma(\cdot)$, the fact that $z\Gamma(z)=\Gamma(z+1)$,  $\frac{2(N-1)}{N-2}=\frac{N}{N-2}+1$ and the characterization of $c_N$ given in \eqref{F:def}. On the other hand, by \eqref{FunSol} and Lemma~\ref{alt:lem},
    \[
    |\Delta G_{B_1}(\cdot,0)|_T 
    =
    |\delta_0|(\Omega)
    =
    \sup\bigg\{-\varphi(0):\varphi\in C^2_c(\Omega),\ |\varphi|_\infty\leq 1\bigg\}=1. \qedhere
    \]
\end{proof}

\begin{lemma}\label{talenti}
Let $\Omega\subseteq\R^N$ be open, bounded, and such that $|\Omega|=|B_1|$. 
For any fixed $x\in\Omega$,
if $G_\Omega(x,\cdot)^\sharp$ denotes the radial symmetric decreasing rearrangement of $G_\Omega(x,\cdot)$, then it holds
\begin{align}\label{hrfwn}
G_\Omega(x,\cdot)^\sharp(y)\leq G_{B_1}(0,y)
\qquad\text{for a.e. }y\in B_1\setminus\{0\}.
\end{align}
\end{lemma}
\begin{proof}
Consider $\psi\in C^\infty_c(B_1)$ to be radially decreasing, and such that $\psi\geq0,|\psi|_1=1$. Construct the sequence
\begin{align*}
\psi_j(y)=j^{N}\psi(j y)
\qquad\text{for any }y\in\R^N,\ j\in\N,
\end{align*}
which is weakly converging to the Dirac measure $\delta_0$. Consider the sequences defined by, for any $j\in\N$,
\begin{align*}
-\Delta u_j=\psi_j \quad\text{in }B_1,
\qquad u_j=0\quad\text{on }\partial B_1,
\end{align*}
and , for $x\in \Omega$ fixed,
\begin{align*}
-\Delta v_j=\psi_j(\cdot-x) \quad\text{in }\Omega,
\qquad v_j=0\quad\text{on }\partial\Omega.
\end{align*}
As $\psi_j(\cdot-x)^\sharp=\psi_j$, by the Talenti's comparison principle \cite{zbMATH03531830}, we deduce
\begin{align}\label{0952385u}
v_j^\sharp\leq u_j
\qquad\text{in }B_1,\text{ for any }j\in\N.
\end{align}
As
\begin{align*}
v_j\longrightarrow G_\Omega(x,\cdot),
\quad u_j\rightarrow G_{B_1}(0,\cdot),
\qquad\text{pointwisely as }j\to+\infty,
\end{align*}
by continuity of the radial symmetric decreasing rearrangement in measure \cite[Theorem 2.4]{zbMATH04132888} we can push inequality \eqref{0952385u} to the limit to get \eqref{hrfwn}.
\end{proof}

\begin{prop}\label{uniquexm}
The function $B_1\ni x\mapsto|G_{B_1}(x,\cdot)|_q^q$ has a maximum point at $x_M=0$.
\end{prop}
\begin{proof}
Let $x\in B_1.$ By the properties of the symmetric decreasing rearrangement and by Lemma~\ref{talenti},
\begin{align*}
|G_{B_1}(x,\cdot)|_q=|G_{B_1}(x,\cdot)^\sharp|_q\leq|G_{B_1}(0,\cdot)|_q,
\end{align*}
as claimed.
\end{proof}

We are ready to show Theorem~\ref{eigenball:thm}.
\begin{proof}[Proof of Theorem~\ref{eigenball:thm}]
This follows from Theorem~\ref{char:thm}, Lemma~\ref{Fnorm:lem}, and Proposition~\ref{uniquexm}. 
\end{proof}

\begin{remark}\label{L11:rmk}
Note that 
\[
\lim_{q\to 1^+}\Lambda_{1,q}(B_1)=2N,
\]
which is, by \cite[Theorem 1.2]{PRT12}, the first eigenvalue of the 1-bilaplacian equation on the unitary ball. Indeed, 
\begin{align*}
\lim_{q\to 1^+}\Lambda_{1,q}(B_1)=
\frac{4\pi^{N/2}}{\Gamma(\frac{N}2-1)}
    \frac{\Gamma(\frac{N}{N-2})\Gamma(\frac{N}2+1)}{\pi^{N/2}\Gamma(\frac{N}{N-2}-1)}
    =4\bigg(\frac{N}{N-2}-1\bigg)\frac{N}2\bigg(\frac{N}2-1\bigg)=2N,
\end{align*}
where we have used several times the recurrence identity $\Gamma(t+1)=t\Gamma(t)$ for $t>0$.
\end{remark}

\begin{proof}[Proof of Proposition~\ref{faber-krahn}]
Exploiting the characterization of $\Lambda_{1,q}(\Omega)$ given in Theorem~\ref{char:thm}, we have, by Lemma~\ref{talenti} and Proposition~\ref{uniquexm},
\begin{align*}
\Lambda_{1,q}(\Omega)=\frac1{|G_\Omega(x_M,\cdot)|_q}
=\frac1{|G_\Omega(x_M,\cdot)^\sharp|_q}
\geq\frac1{|G_{B_1}(0,\cdot)|_q}=\Lambda_{1,q}(B_1).
\end{align*}
\end{proof}

\subsection{Limiting profiles in a ball}\label{sec:limit:ball}

Let $\Omega=B_1$ be the unitary ball of $\R^N$ ($N\geq 3$) centered at the origin, $\a\in(0,\frac{2}{N-2})$, $\b>0$, and let $(U_\b, V_\b)$ be a least-energy solution of 
\begin{align}\label{ball:sys}
-\Delta U_\b = |V_\b|^{\b-1}V_\b\quad \text{ in }B_1,\qquad
-\Delta V_\b = |U_\b|^{\a-1}U_\b\quad \text{ in }B_1,\qquad 
U_\b =V_\b = 0\quad \text{ on }\partial B_1.
\end{align}

In this section, we give an explicit characterization of the limiting profiles $U_\infty$ and $V_\infty$, given by  Theorem~\ref{Uq:conv:thm:intro}, in $B_1$.

\begin{prop}\label{kappa:prop}
There is $\kappa_{N,\a}>0$ such that, as $\beta\to \infty$,
\[
U_\b\to U_\infty=\kappa_{N,\a} G_{B_1}(\cdot ,0)
\quad \text{ in }W^{1,r}_0(B_1)
\text{ for all }r\in [1,1^*),\qquad |\Delta U_\b|_1\to \kappa_{N,\a}.
\]
Moreover, $V_\infty\in W^{2,1+\frac1\a}(B_1)\cap W_0^{1,1+\frac1\a}(B_1)$ is a strong solution of $-\Delta V_\infty = U_\infty^\a$ in $B_1$.
\end{prop}
\begin{proof}
Let $(U_\b, V_\b)$ be a positive least-energy solution of \eqref{ball:sys}. By \eqref{class:sol}, $(U_\b, V_\b)$ is also a classical solution.  A standard symmetrization argument yields that $U_\b$ and $V_\b$ are  radially symmetric, and  decreasing in the radial variable (see for example \cite[Theorem 4.7]{BMT14}). By Theorem~\ref{Uq:conv:thm:intro}, we have that $U_\b\to U_\infty$ in $W^{1,r}(B_1)$ and $V_\b\to V_\infty$ in $C^{0,\sigma}(\overline \Omega)$ as $\b\to\infty$. In particular, $U_\infty$ and $V_\infty$ are also nonnegative, radially symmetric, and decreasing in the radial variable.

Now we claim that $V_\infty\leq 1$ in $B_1$.  If there is $z_0\in B_1$ such that $\eps=V_\infty(z_0)-1>0$, then (by monotonicity and uniform convergence), we have that $V_\beta(x)>1+\frac{\eps}{2}$ for every $x\in B_1$ with $|x|\leq|z_0|$ and for $\b$ sufficiently large. But then $-\Delta U_\b(x) = V_\b(x)^\b\to\infty$ as $\b\to\infty$ for every $x\in B_1$ with $|x|\leq|z_0|$. In particular, if $\varphi\in C^\infty_c(B_1)$ is a nonnegative function such that $\varphi=1$ in $B_{|z_0|}$ and $|\varphi|_\infty\leq 1$, then
    \begin{align*}
    |\Delta U_\infty|_T=\lim_{\b\to\infty}|\Delta U_\b|_T\geq \lim_{\b\to\infty}\int_{B_1}U_\b\,(-\Delta\varphi)
    =\lim_{\b\to\infty}\int_{B_1}(-\Delta U_\b)\,\varphi
    =\infty\qquad \text{ as }\b\to\infty,
    \end{align*}
    which would contradict the fact that $U_\infty\in BL_0(\Omega)$.
    
Next we show that $V_\infty<1$ in $B_1\setminus\{0\}$. Indeed, assume by contradiction that there is $r_0\in(0,1)$ such that $V_\infty=1$ in $B_{r_0}$.  Then, $0=-\Delta V_\infty = U_\infty^\a$ in $B_{r_0}$. Since $U_\infty$ is a nonnegative monotone function, this implies that $U_\infty\equiv 0$ in $B_1$, but this contradicts the fact that $U_\infty$ is nontrivial (see Lemma~\ref{alpha1:lem}).
  
    Since $V_\infty<1$ in $B_1\setminus\{0\}$, this implies that $V_\b^\b\to 0$ locally uniformly in $B_1\setminus\{0\}$ as $\b\to\infty$. Therefore,
    \begin{align*}
    -\Delta U_\b = V_\b^\b\to 0\qquad\text{ locally uniformly in }B_1\setminus\{0\}
    \quad\text{as }\b\to\infty.
    \end{align*}
    As a consequence, for every $\varphi\in C^\infty_c(B_1\backslash \{0\})$,
    \begin{align*}
        \int_\Omega U_\infty\,\Delta\varphi
        =\lim_{\b\to\infty}\int_\Omega U_\b\,\Delta\varphi
        =-\lim_{\b\to\infty}\int_\Omega V_\b^\b\,\varphi =0,
    \end{align*}
namely, the nontrivial Radon measure $-\Delta U_\infty$ has support only on $\{0\}$. Then, for every $\varphi\in C^\infty_c(B_1)$, it holds
\[
\int_{B_1} \varphi \, d(-\Delta U_\infty)=\varphi(0)(-\Delta U_\infty)(B_1),
\]
which implies that $-\Delta U_\infty$ is a multiple of $\delta_0$ and therefore $U_\infty(x)=\kappa_{N,\a}G_{B_1}(\cdot,0)$ in $B_1$ for some $\kappa_{N,\a}>0$, where $G_{B_1}(\cdot,0)$ is given by \eqref{F:def}.
\end{proof}

\begin{lemma}\label{kappa1:lem}
    If $\kappa_{N,\a}$ is as in Proposition~\ref{kappa:prop}, then 
    \begin{align*}
    \kappa_{N,\a}
    = \Lambda_{1,\a+1}^\frac{\a+1}{\a}
    =c_N^{-1-\frac{1}{\a}}\left(
    \frac{\Gamma(\frac{N}{N-2})\Gamma(\frac{N}2+1)}{\pi^{N/2}\Gamma(\frac{2}{N-2}-\a)\,\Gamma(\a+2)}\right)^\frac{1}{\a}.
    \end{align*}
\end{lemma}
\begin{proof}
    By Theorem~\ref{Uq:conv:thm} and Proposition~\ref{kappa:prop}, we have $\Lambda_{1,\a+1}^\frac{1}{\a}u_1=\kappa_{N,\a}G_{B_1}(\cdot, 0)$, so that  
    \[\Lambda_{1,\a+1}^{\frac{\a+1}{\a}}=\Lambda_{1,\a+1}^\frac{1}{\a}|\Delta u_1|_T=\kappa_{N,p} |\Delta G_{B_1}(\cdot,0)|_T=\kappa_{N,p},\]
and the claim now follows from Theorem~\ref{eigenball:thm}.
\end{proof}

Let $G_{B_1}$ denote the Green's function for the Dirichlet Laplacian in the ball $B_1$ and recall that $\kappa_{N,\a}$ is given in Lemma~\ref{kappa1:lem},  and $(U_\b,V_\b)$ denotes a least-energy solution of \eqref{ball:sys}.  We are ready to show Theorem~\ref{ball:thm}.

\begin{proof}[Proof of Theorem~\ref{ball:thm}]
The proof follows from Proposition~\ref{kappa:prop}. Note that the constant $A_{N,\a}$ given in \eqref{k:def} is $A_{N,\a}=\kappa_{N,\a}c_N$, see Lemma~\ref{kappa1:lem}. 
\end{proof}

\begin{remark}
Consider now a general bounded domain $\Omega\subseteq\R^N$ of class $C^{2,\gamma}$ and a least-energy positive solution $(U_\b,V_\b)$ of 
\begin{align*}
-\Delta U_\b = |V_\b|^{\b-1}V_\b\quad \text{ in }\Omega,\qquad
-\Delta V_\b = |U_\b|^{\a-1}U_\b\quad \text{ in }\Omega,\qquad 
U_\b =V_\b = 0\quad \text{ on }\partial \Omega.
\end{align*}
We conjecture that $|V_\b|^{\b-1}V_\b \to C_{N,\a}\delta_{x_M}$ in the sense of distributions as $\b\to\infty$, with $x_M$ as in Corollary~\ref{explicit:cor} and 
\[
C_{N,\a}=|G_\Omega(\cdot,x_M)|_{\a+1}^{-1-\frac1\a}.
\]
Then, $-\Delta U_\infty=C_{N,\a}\delta_{x_M}$ and $U_\infty(x)=C_{N,\a}G_\Omega(x,x_M)$ for $x\in\Omega$, that is, $U_\infty$ is a multiple of the Green's function for the Dirichlet Laplacian in $\Omega$ centered at $x_M$. Then,
\[
V_\infty(x)=C_{N,\a}^\a\int_\Omega G_\Omega(x,y)\,G_\Omega(y,x_M)^\a\;dy
\qquad\text{for }x\in\Omega.
\]
This is a conjecture, because Corollary~\ref{explicit:cor} does not claim that all the minimizers of $\Lambda_{1,p}(\Omega)$ have the same shape (see also Proposition~\ref{prop:uniqueness} for a uniqueness statement in this regard).
\end{remark}

\appendix

\section{Useful results}\label{appendix}
The purpose of this appendix is to give a self-contained description of the space $BL_0(\Omega)$ given by \eqref{eq:BL0}, and some of its properties.  In some cases we write new proofs of known results (Lemma~\ref{lemma:density}) or present new results (Lemma~\ref{lemma:densityC^2}); other short proofs are included for completeness. Remarks~\ref{sense_of_Stampacchia} and~\ref{rem:BP_comment}, on the other hand, are intended to comment some references, while Remark~\ref{nico:rmk} presents a way of introducing the Green function on a Lipchitz domain via regularity results.

Unless otherwise stated, we take $\Omega$ to be a bounded Lipschitz domain, and we recall from \eqref{eq:1*1**} the definitions of $1^*$ and $1^{**}$.

\subsection{General results}

Recall the definitions of $W^{2,p}_\Delta(\Omega)$, for $p\geq 1$, and $BL_0(\Omega)$ given in \eqref{def:Xp} and \eqref{def:BL} respectively, and the total variation $|\Delta \cdot |_T$ of the Laplacian of an $L^1_{loc}(\Omega)$-function, given in \eqref{def:Tnorm}. We use the symbol $\sim$ for the equivalence of two norms. Recall also that the total variation of a Radon measure $\mu$ on $\Omega$ is defined as
\[
|\mu|(\Omega)=\sup \left\{\int_{\Omega} \varphi\, d\mu : \varphi \in C_{c}(\Omega),|\varphi|_{\infty} \leq 1\right\}.
\]

\begin{lemma}[Alternative characterization of $BL_0(\Omega)$]\label{alt:lem}
We have
\[
    BL_0(\Omega)=\left\{u \in W_{0}^{1,1}(\Omega): \Delta u \text{ is a Radon measure with } |\Delta u|(\Omega)<\infty\right\}.
\]  
In particular, for $u\in BL_0(\Omega)$, we have
\begin{align}\label{measures_equiv}
|\Delta u|_T&=|\Delta u|(\Omega)=\sup \left\{\int_{\Omega}  \varphi\,  d\Delta u  : \varphi \in C_c^\infty(\Omega),|\varphi|_{\infty} \leq 1\right\}.
\end{align}
\end{lemma}
\begin{proof}
The alternative characterizations of $BL_0(\Omega)$ and of $|\Delta u|_T$ correspond to \cite[Proposition 2.2]{PRT12}. 
\end{proof}

\begin{lemma}\label{lemma:Deltau}
Let $u\in L^1_{\rm loc}(\Omega)$ be such that $\Delta u\in L^1(\Omega)$. Then
\[
|\Delta u|_1=|\Delta u|_T.
\]
In particular, $W^{2,1}_\Delta(\Omega)\subseteq BL_0(\Omega)$ and the inclusion is strict.
\end{lemma}
\begin{proof}
The proof of the inclusion can be found in \cite[pages 313--314]{PRT12}. It is strict by the following: a solution of $-\Delta u=\delta_y$ in $\Omega$, $u=0$ on $\partial \Omega$, where $\delta_y$ is the Dirac delta concentrated at $y\in \Omega$, belongs to $BL_0(\Omega)$ but not to $W^{2,1}_\Delta(\Omega)$, see \cite[p. 313]{PRT12}.
\end{proof}

\begin{lemma}[Lower-semicontinuity of $|\Delta\cdot|_T$]\label{lemma:lsc} Let $(u_n)_{n\in\N}\subset BL_0(\Omega)$ be a sequence such that 
\[
u_{n} \to u\quad\text{in }L^{1}(\Omega)\text{ as }n\to\infty
\qquad\text{and}\qquad
\sup_{n\in\N}|\Delta u_n|_T<\infty.
\]
Then $u \in B L_{0}(\Omega)$ and
\[
|\Delta u|_T \leq \liminf _{n\to\infty} |\Delta u_n|_T .
\]
\end{lemma}
\begin{proof}
See \cite[Remark 2.1]{PRT12}. For completeness, we include the proof here. Take $\varphi\in C^\infty_c(\Omega)$ with $|\varphi|_\infty\leq 1$. Then
\[
\int_\Omega u\,\Delta \varphi 
=\lim_{n\to\infty}\int_\Omega u_n\,\Delta\varphi 
\leq \liminf_{n\to\infty}|\Delta u_n|_T,
\]
and the result now follows by taking the supremum on $\varphi$.
\end{proof}

Next, we recall a few things about the notion of solution to linear Dirichlet boundary problems with measurable data (see \cite[Definition 3.1]{PonceBook} and references before it).

\begin{definition}[Littman, Stampacchia, and Weinberger \cite{LSW63}]\label{sense_of_Stampacchia} 
Let $\mu$ be a finite Radon measure in $\Omega$. We say that $u\in L^1(\Omega)$ is a very weak solution\footnote{Following the wording of, for example, \cite{zbMATH00840906}.} to
\begin{equation}\label{linear_problem_meas}
-\Delta u=\mu \text{ in } \Omega,\qquad u=0 \text{ on }\quad\partial\Omega,
\end{equation}
if, for every $\varphi\in C^\infty_0(\overline \Omega)=\{\zeta\in C^\infty(\overline\Omega):\ \zeta=0\text{ on }\partial\Omega\}$,
\begin{equation}\label{eq:notion_solution}
\int_\Omega \varphi\, d\mu= -\int_\Omega u\,\Delta \varphi.
\end{equation}
\end{definition}

\begin{remark}
The authors of \cite{PRT12} call this a solution \emph{in the sense of Stampacchia}, but take, instead, test functions $\varphi\in W^{1,2}_0(\Omega)\cap C(\overline \Omega)$ such that $\Delta \varphi\in C(\overline \Omega)$. 
\end{remark}

\begin{definition}[Distributional solutions]
Let $\mu$ be a finite Radon measure in $\Omega$. We say that $u\in L^1_{loc}(\Omega)$ is a distributional solution to
\[
-\Delta u=\mu \text{ in } \Omega
\]
if \eqref{eq:notion_solution} holds for every $\varphi\in C^\infty_c(\Omega)$.
\end{definition}

Related to this, we have the following result that was shown in \cite[Theorem 5.1]{LSW63} and \cite[Theor\`eme 9.1]{Stampacchia} (see also \cite[Proposition 5.1]{PonceBook}).

\begin{theorem}\label{th:equivalence_solutions} Let $\mu$ be a finite Radon measure on $\Omega$. Then there exists exactly one very weak solution $u$ of \eqref{linear_problem_meas}.
Moreover, for every $q\in [1,1^*)$, we have $u\in W^{1,q}_0(\Omega)$ and there exists a universal constant $C>0$ such that
\[
\|u\|_{W^{1,q}_0}\leq C |\mu|(\Omega).
\]
\end{theorem}

\begin{remark}\label{nico:rmk}
We note that Theorem~\ref{th:equivalence_solutions} can be used to build the Green function $G_\Omega$ of a Lipschitz domain $\Omega$, and to deduce some of its properties. Indeed, for every fixed $y\in\Omega$, $G(\cdot, y)$ can be defined as the unique very weak solution to 
\[
-\Delta u=\delta_y \text{ in } \Omega,\qquad u=0 \text{ on }\partial\Omega,
\]
where $\delta_y$ stands for the Dirac delta centred at $y$. Then, by Definition~\ref{sense_of_Stampacchia}, for every $\varphi\in C^\infty_0(\overline{\Omega})$ it holds
\begin{align}\label{reprrrr}
\varphi(y)=-\int_\Omega G_\Omega(x,y)\,\Delta\varphi(x)\;dx,
\end{align}
which is a representation formula for $\varphi$.
Take now $\varphi,\psi\in C^\infty_0(\overline{\Omega})$ such that $\Delta\varphi,\Delta\psi\in C^\infty_c(\Omega)$: then, on the one hand, using \eqref{reprrrr} above, one obtains
\[
-\int_\Omega\varphi(y)\,\Delta\psi(y)\;dy=-\int_\Omega\bigg(\int_\Omega G_\Omega(x,y)\,\Delta\psi(y)\;dy\bigg)\Delta\varphi(x)\;dx,
\]
while, on the other hand, using \eqref{reprrrr} on $\psi$,
\[
-\int_\Omega\varphi(x)\,\Delta\psi(x)\;dx=-\int_\Omega\psi(x)\,\Delta\varphi(x)\;dx=-\int_\Omega\bigg(\int_\Omega G(y,x)\,\Delta\psi(y)\;dy\bigg)\Delta\varphi(x)\;dx.
\]
From this, one deduces 
\[
\int_\Omega G_\Omega(x,y)\,\Delta\psi(y)\;dy=\int_\Omega G_\Omega(y,x)\,\Delta\psi(y)\;dy
\qquad\text{for a.e. }x\in\Omega.
\]
and therefore
\[
G_\Omega(x,y)=G_\Omega(y,x) 
\qquad\text{for a.e. }x,y\in\Omega.
\]
Furthermore, the unique very weak solution to \eqref{linear_problem_meas} can be represented by
\[
u(x)=\int_\Omega G_\Omega(x,y)\,d\mu(y)
\qquad\text{for a.e. }x\in\Omega,
\]
because, for every $\varphi\in C^\infty_0(\overline\Omega)$, using \eqref{eq:notion_solution}, \eqref{reprrrr}, and Fubini's theorem:
\[
-\int_\Omega u\,\Delta \varphi=\int_\Omega\varphi\;d\mu=-\int_\Omega\bigg(\int_\Omega G_\Omega(x,y)\;d\mu(y)\bigg)\;\Delta\varphi(x)\;dx.
\]
\end{remark}

The following result shows the equivalence between the two notions of solution under geometric or regularity assumptions on $\Omega$.

\begin{prop}\label{prop:equivalence_solutions}
Let $\Omega$ be a bounded domain, which is either convex or of class $C^{1,\gamma}$, for some $\gamma\in (0,1]$. Take a finite Radon measure $\mu$ on $\Omega$. Then $u \in W^{1,1}_0(\Omega)$ is a distributional solution of \eqref{linear_problem_meas} if, and only if, $u\in L^1(\Omega)$ is a very weak solution of \eqref{linear_problem_meas}. 
\end{prop}
\begin{proof}
  Let  $u\in L^1(\Omega)$ be a solution of \eqref{linear_problem_meas} in the sense of Definition~\ref{sense_of_Stampacchia} (i.e., a very weak solution). Then $u\in W^{1,1}_0(\Omega)$, by Theorem~\ref{th:equivalence_solutions}.  Moreover, $u$ is a distributional solution since $C^\infty_c(\Omega)\subseteq C^{\infty}_0(\overline \Omega)$.

  The converse statement, on the other hand, is the content of \cite[Proposition 4.3]{PRT12}.
\end{proof}

\begin{coro}\label{BL0_inequality}
    Let $\Omega$ be a domain, either convex or of class $C^{1,\gamma}$, for some $\gamma\in (0,1]$. Given $q\in [1,1^*)$, there exists $C>0$ such that
    \begin{equation}\label{eq:inequality_BL0}
    \|u\|_{W^{1,q}_0} \leq C |\Delta u|_T \qquad \text{for any } u\in BL_0(\Omega).
    \end{equation}
\end{coro}
\begin{proof}
Any $u\in BL_0(\Omega)$ is a $W^{1,1}_0(\Omega)$--distributional solution of the equation
\begin{equation}\label{eq:dist_aux}
\Delta w=\Delta u \text{ in } \Omega.
\end{equation}
Indeed, for $\varphi\in C^\infty_c(\Omega)$, by the embedding of the space of Radon measures in the space of distributions and since $\Delta u\in \mathcal{D}'(\Omega)$ is a Radon measure:
\[
\int_\Omega \Delta u \varphi = \langle \Delta u,\varphi\rangle_{\mathcal{D}'(\Omega)\times \mathcal{D}(\Omega)}=\int_\Omega u\Delta \varphi.
\]
Then, by Theorem~\ref{th:equivalence_solutions} and Proposition~\ref{prop:equivalence_solutions}, $u$ is the (unique) very weak solution of \eqref{eq:dist_aux} with zero Dirichlet boundary conditions. The result now follows directly from Theorem~\ref{th:equivalence_solutions} and \eqref{measures_equiv}.
\end{proof}

\begin{remark}\label{rem:inequality}
    In \cite[p.  318]{PRT12}, the authors give an example of a Lipschitz bounded domain in $\R^2$ and of a \emph{nontrivial} function $u\in BL_0(\Omega)$ such that $\Delta u=0$ in the classical sense. Then, in particular, $u$ is a distributional solution of
    \begin{equation}\label{eq:harmonic}
    -\Delta u=0 \text{ in } \Omega.
    \end{equation}
    Therefore, it is not a very weak solution of \eqref{eq:harmonic} with $u=0$ on $\partial \Omega$(which is unique and is the trivial one, by Theorem~\ref{th:equivalence_solutions}). This also shows that an inequality like \eqref{eq:inequality_BL0} may not hold for a general Lipschitz domain, and more regularity  is needed (or convexity).
\end{remark}

\begin{lemma}[Equivalent norms]\label{equiv:lem} The following are Banach spaces:
\begin{enumerate}
\item ($\Omega$ is of class $C^{1,1}$) $W^{2,p}_\Delta(\Omega)$ ($p>1$),  when endowed with 
$|\Delta u |_p\sim \|u\|_{W^{1,p}}+|\Delta u|_p\sim \|u\|_{W^{2,p}}$. In particular, $W^{2,p}_\Delta(\Omega)=W^{2,p}(\Omega)\cap W^{1,p}_0(\Omega)$.

\item ($\Omega$ convex, or $C^{1,\gamma}$, for some $\gamma\in (0,1]$) $W^{2,1}_\Delta(\Omega)$,  when endowed with $|\Delta u |_1\sim \|u\|_{W^{1,1}}+|\Delta u|_1$.
\item ($\Omega$ convex, or $C^{1,\gamma}$, for some $\gamma\in (0,1]$) $BL_0(\Omega)$, when endowed with $|\Delta u |_T\sim \|u\|_{W^{1,1}}+|\Delta u|_T$.
\end{enumerate}
\end{lemma}
\begin{proof}
 1. From elliptic regularity theory (see for instance  in \cite[Lemma 9.17]{GT98}), we have the existence of $C>0$ such that
\[
\|u\|_{W^{2,p}}\leq C |\Delta u|_p \qquad \text{for every }u\in W^{2,p}(\Omega)\cap W^{1,p}_0(\Omega).
\] 
and so the conclusion follows.

 3. As for $BL_0(\Omega)$, the fact that it is a Banach space when endowed with $\|u\|_{W^{1,1}(\Omega)}+|\Delta u|_T$ follows from \cite[Proposition 2.3]{BP18}. We include here the proof for completeness. Take a Cauchy sequence ${(u_n)}_{n\in\N}\subseteq BL_0(\Omega)$. Then ${(u_n)}_{n\in\N}$ is a Cauchy sequence in $W^{1,1}_0(\Omega)$; since this space is complete, there exists $u\in W^{1,1}_0(\Omega)$ such that 
 \[
 u_n \to u \quad \text{ in } W^{1,1}_0(\Omega), \text{ hence also in } L^1(\Omega).
 \]
 Given $\eps>0$, take $\bar n\in\N$ such that
 \[
 |\Delta u_n-\Delta u_m|_T<\eps \qquad \text{ for every $n,m\geq \bar n$.} 
 \]
 Since, for each $n\geq \bar n$, we have $u_n-u_m\in BL_0(\Omega)$, $u_n-u_m\to u_n-u$ as $m\to \infty$ in $L^1(\Omega)$ and $\sup_{m\geq \bar n}|\Delta (u_n-u_m)|_T<\infty$, then by Lemma~\ref{lemma:lsc} we have $u_n-u\in BL_0(\Omega)$ (so that also $u\in BL_0(\Omega)$) and
 \[
 |\Delta u_n-\Delta u|_T=|\Delta (u_n-u)|_T\leq \lim_{m\to\infty} |\Delta (u_n-u_m)|_T\leq \varepsilon.
 \]
 Therefore, $|\Delta u_n-\Delta u|_T\to 0$ as $n\to\infty$.
 
 The equivalence of the norms $|\Delta u |_T\sim \|u\|_{W^{1,1}}+|\Delta u|_T$ is a direct consequence of Corollary~\ref{BL0_inequality}, which yields (for $q=1$):
 \begin{equation}\label{eq:Stampacchia}
 \|u\|_{W^{1,1}}\leq C |\Delta u|_T.
 \end{equation}

\smallbreak

2.  As for $W^{2,1}_\Delta(\Omega)$, the  equivalence of the norms is a consequence of \eqref{eq:Stampacchia} together with the fact that $|\Delta u|_1=|\Delta u|_T$ when $u\in W^{2,1}_\Delta(\Omega)$ (Lemma~\ref{lemma:Deltau}). The fact that $W^{2,1}_\Delta(\Omega)$ is a Banach space is shown in \cite[Proposition 11]{CassaniRufTarsi}, but we include a proof for completeness: taking a Cauchy sequence ${(u_n)}_{n\in\N}$ in $W^{2,1}_\Delta(\Omega)$, then ${(u_n)}_{n\in\N}$ is a Cauchy sequence in $W^{1,1}_0(\Omega)$ and ${(\Delta u_n)}_{n\in\N}$ is a Cauchy sequence in $L^1(\Omega)$. Then there exist $u\in W^{1,1}_0(\Omega)$ and $v\in L^1(\Omega)$ such that
\[
u_n\to u \quad\text{in } W^{1,1}(\Omega),\qquad \Delta u\to v \quad\text{in } L^1(\Omega).
\]
Given $\varphi\in C^\infty_c(\Omega)$, we have
\[
\int_\Omega u \,\Delta \varphi=\lim_{n\to\infty} \int_\Omega u_n\, \Delta \varphi=\lim_{n\to\infty} \int_\Omega \Delta u_n \,\varphi=\int_\Omega v \varphi,
\]
so $\Delta u=v\in L^1(\Omega)$ and $u_n\to u$ in $W^{2,1}_\Delta(\Omega)$.
\end{proof}

\subsection{Embeddings}

For $N\geq 3$, if $\Omega$ is a bounded set, then
\[
W^{1,1}_0(\Omega)\hookrightarrow L^{t}(\Omega) \text{ is continuous for $t\in [1,1^*]$, compact for $t\in [1,1^*)$},
\]
where $1^*$ is defined as in \eqref{eq:1*1**}. If $\Omega$ is a bounded Lipschitz domain, then
\[
W^{2,1}(\Omega)\hookrightarrow L^{t}(\Omega)\text{ is continuous for $t\in [1,1^{**}]$, compact for $t\in [1,1^{**})$},
\]
where $1^{**}$ is defined as in \eqref{eq:1*1**}. See for instance \cite[Theorems 7.22 and 7.26]{GT98}. The last embedding also holds true for $W^{2,1}(\Omega)\cap W^{1,1}_0(\Omega)$, which is a closed subset of $W^{2,1}(\Omega)$.

It is also useful to recall the general case:
\begin{lemma}\label{lemma:Sobolev_embedding} The following hold true when $\Omega$ is a bounded Lipschitz domain:
\begin{enumerate}
\item For $2p<N$, $W^{2,p}(\Omega)\hookrightarrow L^t(\Omega)$ continuous if $t\in  [1,\frac{Np}{N-2p}]$, compact if $t\in [1,\frac{Np}{N-2p})$;
\item For $2p=N$, $W^{2,p}(\Omega)\hookrightarrow L^t(\Omega)$ is compact for every $t\in [1,\infty)$;
\item For $2p>N$, $W^{2,p}(\Omega)\hookrightarrow C^{m,\gamma}(\overline \Omega)$ is continuous, where $m$ is the largest positive integer such that $\gamma=2-\frac{N}p-m\in (0,1)$, and $W^{2,\beta}(\Omega)\hookrightarrow C^{m',\gamma'}(\overline \Omega)$ is compact for $m'\leq m$, $\gamma'\leq \gamma$ and either $m'<m$ or $\gamma'<\gamma$.
\end{enumerate}
\end{lemma}
We recall the definition of the weak-$L^q$ spaces, which are nothing else than the Lorentz spaces $L^{q,\infty}(\Omega)$. Given a measurable function $u:\Omega\to \R$, its distribution function $\mu_u:\R^+\to \R$ is given by $\mu_u(t):=|\{x\in \Omega:\ |u(x)|>t\}|$. For $q\geq 1$, we define
\[
\|u\|_{q,\infty}^q:=\sup_{t>0} t^q \mu_u(t),\quad L^{q,\infty}(\Omega):=\{u:\Omega\to \R \text{ measurable}:\ \|u\|_{q,\infty}<\infty\}.
\]
Recall that $L^p(\Omega)\hookrightarrow L^{p,\infty}(\Omega)$ and $L^{p,\infty}(\Omega)\hookrightarrow L^q(\Omega)$, for $1\leq q<p$, are continuous embeddings (see for instance \cite[Section 1.4]{Grafakos}).
\begin{prop}[Embeddings]\label{compact:prop}
Let $\Omega$ be either a convex set or a set of class $C^{1,\gamma}$, for some $\gamma\in (0,1]$. The following embedding is continuous
\[
B L_{0}(\Omega) \hookrightarrow L^{1^{**},\infty}(\Omega).
\]
Moreover, the following embeddings are compact:
\begin{align*}
BL_{0}(\Omega) &\hookrightarrow W^{1,q}_0(\Omega) && \text{for } q\in [1,1^*),\\
BL_{0}(\Omega) &\hookrightarrow L^r(\Omega) && \text{for } r\in [1,1^{**}),\\ 
BL_0(\Omega)\cap L^\infty(\Omega) &\hookrightarrow W_0^{1,q}(\Omega) && \text{for }  q\in [1,2).
\end{align*}
\end{prop}
\begin{proof}
The assumptions on $\Omega$ allow to use both Corollary~\ref{BL0_inequality} and other estimates for elliptic problems with measure data.

Indeed, the continuity of the embedding $B L_{0}(\Omega) \hookrightarrow L^{1^{**},\infty}(\Omega)$ follows from \cite[Proposition 5.7]{PonceBook} (see also \cite[Theorem 3]{CassaniRufTarsi} for the optimal constant).

The proof of the compactness of the embedding $BL_{0}(\Omega) \hookrightarrow L^r(\Omega)$ can be found in \cite[Proposition 5.9]{PonceBook}, and the one of $BL_{0}(\Omega) \hookrightarrow W^{1,q}_0(\Omega)$ in \cite[Proposition 5.10]{PonceBook}.

Finally, the fact that $BL_0(\Omega)\cap L^\infty(\Omega)\hookrightarrow W_0^{1,q}(\Omega)$ is compact follows from the interpolation inequalities
\[
\|\nabla u\|_{L^2(\Omega)}\leq \|u\|_{L^\infty}|\Delta u|_T \quad \text{ for every } u\in BL_0(\Omega)\cap L^\infty(\Omega)
\]
(see \cite[Lemma 5.8]{PonceBook}), $\|\nabla u\|_{L^q(\Omega)}\leq \|\nabla u\|_1^{\frac{2}{q}-1} \|\nabla u\|_2^{2-\frac{2}{q}}$ and the compact embedding $BL_0(\Omega)\hookrightarrow W^{1,1}_0(\Omega)$.
\end{proof}
\begin{remark}\label{rem:compactness}
Since $W^{2,1}_\Delta \hookrightarrow BL_0(\Omega)$ is continuous (Lemma~\ref{lemma:Deltau}), the results of Proposition~\ref{compact:prop} are true with $BL_0(\Omega)$ replaced by $W^{2,1}_\Delta(\Omega)$.
\end{remark}


\subsection{Density}

\begin{lemma}[Density]\label{density:lemma}
Let $u \in B L_{0}(\Omega)$. Then there exists a sequence ${(u_n)}_{n\in\N}\subseteq C^{\infty}(\Omega) \cap C(\overline{\Omega}) \cap B L_{0}(\Omega)$ converging strictly to $u$, that is:
\begin{align*}
u_n \to u \quad \text{ strongly in } W^{1,1}_0(\Omega),\qquad
|\Delta u_n|_T\to |\Delta u|_T.
\end{align*}
\end{lemma}
\begin{proof}
See \cite[Proposition 3.2]{PRT12}.
\end{proof}

We also mention that the closure of $C^\infty_c(\Omega)$ with respect to the norm $|\Delta \cdot|_1$ is denoted $W^{2,1}_{\Delta,0}(\Omega)$ and it is studied in \cite{PRT15} in the context of eigenvalue problems.

\begin{lemma}\label{lemma:density}
If $u\in  C^\infty(\Omega) \cap C(\overline \Omega)\cap BL_0(\Omega)$, then $\Delta u\in L^1(\Omega)$. In particular,
\[
C^\infty(\Omega)\cap C(\overline \Omega)\cap BL_0(\Omega)\subseteq W^{2,1}_\Delta(\Omega).
\]
\end{lemma}
\begin{proof}
This is stated without proof in \cite[page 317]{PRT12}. We include here a proof for completeness.

Given $n\in \N$ sufficiently large, take $\varphi_n\in C^\infty_c(\Omega)$ such that 
\[
\varphi_n=1\text{ in the set } \{x\in \Omega:\ \dist(x,\partial \Omega)\geq 1/n\}.
\]
 Define the function
\[
v_n=\frac{(\text{sign}(\Delta u)\ast \eta_{1/n}) \varphi_n}{|(\text{sign}(\Delta u)\ast \eta_{1/n}) \varphi_n|_\infty}\in C^\infty_c(\Omega),
\]
where $\eta_{1/n}$ is a sequence of mollifiers. Then $\|v_n\|_\infty=1$, $v_n\to \text{sign}(\Delta u)$ a.e. in $\Omega$. By Fatou's lemma ($v_n \Delta u$ is bounded and $|\Omega|<\infty$),
\[
\int_\Omega |\Delta u|=\int_\Omega \lim_n  \Delta uv_n\leq \liminf_n \int \Delta u v_n=\liminf_n \int u \Delta v_n\leq |\Delta u|_T<\infty.
\]
\end{proof}

\begin{lemma}\label{lemma:densityC^2} Let $\Omega$ be a bounded domain of class $C^{2,\gamma}$, for some $\gamma\in (0,1]$. Then the space $C^{2,\gamma}(\overline \Omega)$ is dense in $W^{2,1}_\Delta (\Omega)$.
\end{lemma}
\begin{proof}Consider $u\in W^{2,1}_\Delta(\Omega)$, which in particular means that $\Delta u\in L^1(\Omega)$, and fix $\eps>0$. There exists $f\in C^\infty_c(\Omega)\subset C^{2,\gamma}(\overline \Omega)$ such that $|f-\Delta u|_1<\eps$. If $v$ denotes the solution to the Dirichlet problem $\Delta v=f$ in $\Omega$ and $v=0$ on $\partial\Omega$ then, by elliptic global regularity (see \cite[Theorem 6.14]{GT98}), $v\in C^{2,\gamma}(\overline{\Omega})$ and 
        \[
       |\Delta v-\Delta u|_1=|f-\Delta u|_1<\eps,
        \]
        by construction. Recalling that  $|\Delta \cdot|_1$ is a norm in $W^{2,1}_\Delta(\Omega)$ (by Lemma~\ref{equiv:lem}), we see that $C^\infty(\overline\Omega)$ is dense in $W^{2,1}_\Delta(\Omega)$.
        \end{proof}

\begin{remark}\label{rem:BP_comment}
Proposition 2.7 in \cite{BP18} states that $BL_0(\Omega)\hookrightarrow L^{1^{**}}(\Omega)$ is continuous, among other results. This is not true at least for $N\geq 3$, as it can be shown by a simple counterexample: on $\Omega=B_1$, we have $G_{B_1}(\cdot,0)\in BL_0(B_1)$ with $|\Delta G_{B_1}(\cdot,0)|_T=|\delta_0|_T=1$ and, more explicitly,
\[
G_{B_1}(x,0)=c_N(|x|^{2-N}-1)
\qquad\text{for }x\in B_1, 
\]
which clearly does not lie in $L^{1^{**}}(B_1)$ (see also the counterexample in \cite[page 83]{PonceBook}).
This also implies that
    \[
    C^\infty(\Omega)\cap C(\overline \Omega)   \cap BL_0(\Omega) \text{ is } not \text{ contained in } W^{2,1}(\Omega). 
    \]
    In fact, if this were true, then the proof of \cite[Proposition 2.7]{BP18} would be correct, showing the wrong fact that $BL_0(\Omega)\hookrightarrow L^{1^{**}}(\Omega)$ is continuous. The ``proof'' would go  as follows: given $u\in BL_0(\Omega)$, by Lemma~\ref{density:lemma} we can take a sequence $(u_n)\subset C^\infty(\Omega)\cap C(\overline \Omega)   \cap BL_0(\Omega)$ such that $u_n\to u$ strictly in $BL_0(\Omega)$. If $C(\overline \Omega) \cap C^\infty(\Omega) \cap BL_0(\Omega)\subset W^{2,1}(\Omega)$, then $(u_n)\subset W^{2,1}(\Omega)$, which is continuously embedded in $L^{1^{**}}(\Omega)$ and  the rest of the proof easily follows. 
    
    Although this claim in \cite[Proposition 2.7]{BP18} is not true, we emphasize that this fact is never used in the proof of the main theorems, which remain valid. Since the paper deals with subcritical nonlinearities, only the compact embeddings presented in our Proposition~\ref{compact:prop} are needed.
    
    We also point out that the same incorrect fact is also used  in \cite[Proposition 2.7]{BP18} to prove the compactness of $BL_0(\Omega)\hookrightarrow L^q(\Omega)$ for $q\in [1,1^{**})$; although the proof contains a mistake, the result is correct, see our Proposition~\ref{compact:prop}.
\end{remark}

\vspace{3\bigskipamount}

\noindent\textbf{Nicola Abatangelo}\\
Dipartimento di Matematica\\ 
Alma Mater Studiorum Università di Bologna\\ 
P.zza di Porta S. Donato, 5\\
40126 Bologna, Italy\\
\texttt{nicola.abatangelo@unibo.it}
\bigskip

\noindent\textbf{Alberto Salda\~na}\\
Instituto de Matem\'aticas\\
Universidad Nacional Aut\'onoma de M\'exico\\
Circuito Exterior, Ciudad Universitaria\\
04510 Coyoac\'an, Ciudad de M\'exico, Mexico\\
\texttt{alberto.saldana@im.unam.mx}
\bigskip

\noindent\textbf{Hugo Tavares}\\
CAMGSD - Centro de An\'alise Matem\'atica, Geometria e Sistemas Din\^amicos\\
Departamento de Matem\'atica do Instituto Superior T\'ecnico\\
Universidade de Lisboa\\
1049-001 Lisboa, Portugal\\
\texttt{hugo.n.tavares@tecnico.ulisboa.pt}

\end{document}